\documentclass{gtmon_a}
\pdfoutput=1

\usepackage{graphicx}

\proceedingstitle{The Zieschang Gedenkschrift}
\conferencestart{5 September 2007}
\conferenceend{8 September 2007}
\conferencename{Conference in honour of Heiner Zieschang}
\conferencelocation{Toulouse, France}

\editor{Michel Boileau}
\givenname{Michel}
\surname{Boileau}

\editor{Martin Scharlemann}
\givenname{Martin}
\surname{Scharlemann}

\editor{Richard Weidmann}
\givenname{Richard}
\surname{Weidmann}

\title{Epimorphisms between $2$--bridge link groups}

\author{Tomotada Ohtsuki}
\givenname{Tomotada}
\surname{Ohtsuki}
\address{Research Institute for Mathematical Sciences\\
Kyoto University\\\newline
Sakyo-ku\\Kyoto\\ 606-8502\\
Japan}
\email{tomotada@kurims.kyoto-u.ac.jp}
\urladdr{http://www.kurims.kyoto-u.ac.jp/~tomotada/}

\author{Robert Riley}
\givenname{Robert}
\surname{Riley}
\address{Passed away on March 4, 2000}

\author{Makoto Sakuma}
\givenname{Makoto}
\surname{Sakuma}
\address{ Department of Mathematics\\
 Graduate School of Science\\
 Hiroshima University\\\newline
 Higashi-Hiroshima\\ 739-8526\\
 Japan}
\email{sakuma@math.sci.hiroshima-u.ac.jp}
\urladdr{http://www.math.sci.hiroshima-u.ac.jp/~sakuma/}

\dedicatory{Dedicated to the memory of Professor Heiner Zieschang}

\volumenumber{14}
\issuenumber{}
\publicationyear{2008}
\papernumber{019}
\startpage{417}
\endpage{450}

\doi{}
\MR{}
\Zbl{}

\arxivreference{} 

\keyword{2-bridge knot}
\keyword{knot group}
\keyword{character variety}
\keyword{degree 1 map}
\subject{primary}{msc2000}{57M25}
\subject{secondary}{msc2000}{57M05}
\subject{secondary}{msc2000}{57M50}

\received{30 May 2006}
\revised{15 March 2007}
\accepted{15 March 2007}
\published{29 April 2008}
\publishedonline{29 April 2008}
\proposed{}
\seconded{}
\corresponding{}
\version{}

\makeatletter
\def\cnewtheorem#1[#2]#3{\newtheorem{#1}{#3}[section]
\expandafter\let\csname c@#1\endcsname\c@Theorem}
\makeatother

\AtBeginDocument{\let\bar\wbar\let\tilde\wtilde\let\hat\what}
\makeop{id}
\makeop{tr}
\def\hgr{\smash{\hat\Gamma_r}}

\newtheorem{Theorem}{Theorem}[section]
\cnewtheorem{Proposition}[Theorem]{Proposition}
\cnewtheorem{Lemma}[Theorem]{Lemma}
\cnewtheorem{Sublemma}[Theorem]{Sublemma}
\cnewtheorem{Corollary}[Theorem]{Corollary}
\cnewtheorem{Claim}[Theorem]{Claim}

\theoremstyle{definition}
\cnewtheorem{Definition}[Theorem]{Definition}
\cnewtheorem{Example}[Theorem]{Example}
\cnewtheorem{Convention}[Theorem]{Convention}
\cnewtheorem{Remark}[Theorem]{Remark}
\cnewtheorem{Question}[Theorem]{Question}
\cnewtheorem{Conjecture}[Theorem]{Conjecture}
\makeautorefname{Example}{Example}
\makeautorefname{Theorem}{Theorem}
\makeautorefname{Proposition}{Proposition}
\makeautorefname{Remark}{Remark}
\makeautorefname{Corollary}{Corollary}
\makeautorefname{Lemma}{Lemma}
\makeautorefname{Question}{Question}

\newcommand{\interiorop}{\operatorname{int}}

\newcommand{\lk}{\operatorname{lk}}

\newcommand{\ZZ}{\mathbb{Z}}
\newcommand{\QQ}{\mathbb{Q}}
\newcommand{\RR}{\mathbb{R}}
\newcommand{\CC}{\mathbb{C}}
\newcommand{\HH}{\mathbb{H}}
\newcommand{\QQQ}{\hat{\mathbb{Q}}}

\newcommand{\Conway}{\boldsymbol{S}^{\!2}}
\newcommand{\Conways}
{(\boldsymbol{S}^{\!2},\boldsymbol{P})}
\newcommand{\PP}{\boldsymbol{P}}
\newcommand{\PConway}{\boldsymbol{S}}

\newcommand{\rtangle}[1]{(B^3,t({#1}))}

\newcommand{\RX}{\mbox{$\mathcal{X}$}}

\newcommand{\DD}{\mathcal{D}}
\newcommand{\RGPP}[1]{\smash{\hat\Gamma_{#1}}}
\newcommand{\RGP}[1]{\Gamma_{#1}}

\newcommand{\cfr}{\mbox{\boldmath$a$}}

\newcommand{\svert}{\,|\,}

\newcommand{\llangle}{\langle\langle}
\newcommand{\rrangle}{\rangle\rangle}

\newcommand{\fns}{\small}
\newcommand{\lra}{\longrightarrow}
\newcommand{\ol}{\wwbar}

\newcommand{\pc}[2]{\mbox{$\begin{array}{c}
    \includegraphics[scale=#2]{\figdir/#1}
    \end{array}$}}

\begin{document}

\begin{asciiabstract}    
We give a systematic construction of epimorphisms
between 2-bridge link groups.
Moreover, we show that 2-bridge links having
such an epimorphism between their link groups
are related by a map between the ambient spaces
which only have a certain specific kind of singularity.
We show applications of these epimorphisms
to the character varieties for 2-bridge links
and pi_1-dominating maps among 3-manifolds.
\end{asciiabstract}

\begin{htmlabstract}
We give a systematic construction of epimorphisms
between 2&ndash;bridge link groups.
Moreover, we show that 2&ndash;bridge links having
such an epimorphism between their link groups
are related by a map between the ambient spaces
which only have a certain specific kind of singularity.
We show applications of these epimorphisms
to the character varieties for 2&ndash;bridge links
and &pi;<sub>1</sub>&ndash;dominating maps among 3&ndash;manifolds.
\end{htmlabstract}

\begin{abstract}    
We give a systematic construction of epimorphisms
between $2$--bridge link groups.
Moreover, we show that $2$--bridge links having
such an epimorphism between their link groups
are related by a map between the ambient spaces
which only have a certain specific kind of singularity.
We show applications of these epimorphisms
to the character varieties for $2$--bridge links
and $\pi_1$--dominating maps among $3$--manifolds.
\end{abstract}

\maketitle


\section{Introduction}
\label{Sec:Introduction}

For a knot or a link, $K$, in $S^3$,
the fundamental group $\pi_1(S^3-K)$
of the complement is called
the \textit{knot group} or the \textit{link group}
of $K$, and is denoted by $G(K)$.
This paper is concerned with the following problem.
\begin{center}\begin{minipage}{26pc}
{\sl For a given knot (or link) $K$,
characterize which knots (or links) $\tilde{K}$
admit an epimorphism $G(\tilde{K}) \to G(K)$.}
\end{minipage}\end{center}
This topic has been studied in various places in the literature,
and, in particular,
a complete classification has been obtained
when $K$ is the $(2,p)$ torus knot and $\tilde{K}$ is a $2$--bridge knot,
and when $K$ and $\tilde{K}$ are prime knots with up to 10 crossings;
for details, see \fullref{sec.ekg}.
A motivation for considering such epimorphisms is that
they induce a partial order on the set of prime knots
(see \fullref{sec.ekg}),
and we expect that new insights into the theory of knots
may be obtained in the future by studying such a structure,
in relation with topological properties and
algebraic invariants of knots related to knot groups.

In this paper, we give
a systematic construction of epimorphisms
between\linebreak $2$--bridge link groups.
We briefly review $2$--bridge links; for details
see \fullref{Sec:tangle}.\linebreak
For $r \in \QQQ:=\QQ\cup\{\infty\}$,
the $2$--bridge link $K(r)$
is the link obtained by gluing two trivial $2$--component tangles in $B^3$
along $(S^2, \mbox{$4$ points})$
where the loop in $S^2-(\mbox{$4$ points})$
of slope $\infty$ is identified with that of slope $r$,
namely
the double cover of the gluing map
\big($\in {\rm Aut}(T^2) = SL(2,\ZZ)$\big)
takes $\infty$ to $r$,
where $SL(2,\ZZ)$ acts on $\QQQ$ by the linear fractional 
transformation.
To be more explicit,
for a continued fraction expansion
\begin{center}\begin{picture}(230,70)
\put(0,48){$\displaystyle{
r=[a_1,a_2, \cdots,a_{m}] =
\cfrac{1}{a_1+
\cfrac{1}{ \raisebox{-5pt}[0pt][0pt]{$a_2 \, + \, $}
\raisebox{-10pt}[0pt][0pt]{$\, \ddots \ $}
\raisebox{-12pt}[0pt][0pt]{$+ \, \cfrac{1}{a_{m}}$}
}} \ ,}$}
\end{picture}\end{center}
a plat presentation of $K(r)$ is given as shown in \fullref{fig.fd}.

We give a systematic construction of epimorphisms
between $2$--bridge link groups
in the following theorem,
which is proved in \fullref{Sec:algebraic proof}.

\begin{Theorem}
\label{Thm:epimorphism1}
There is an epimorphism from
the $2$--bridge link group $G(K(\tilde r))$ to
the $2$--bridge link group $G(K(r))$,
if $\tilde r$ belongs to the $\RGPP{r}$--orbit of $r$ or $\infty$.
Moreover the epimorphism sends the upper meridian pair of
$K(\tilde r)$ to that of $K(r)$.
\end{Theorem}

Here, we define the $\hgr$--action on $\QQQ$ below,
and we give the definition of an upper meridian pair
in \fullref{Sec:tangle}.
For some simple values of $r$,
the theorem is reduced to
Examples \ref{ex.r=infty}--\ref{ex.r=half_integer}.

The $\hgr$--action on $\QQQ$ is defined as follows.
Let $\DD$ be the \textit{modular tessellation}, that is,
the tessellation of the upper half
space $\HH^2$ by ideal triangles which are obtained
from the ideal triangle with the ideal vertices $0, 1,
\infty \in \QQQ$ by repeated reflection in the edges.
Then $\QQQ$ is identified with the set of the ideal vertices of $\DD$.
For each $r\in \QQQ$,
let $\RGP{r}$ be the group of automorphisms of
$\DD$ generated by reflections in the edges of $\DD$
with an endpoint $r$.
It should be noted that $\RGP{r}$
is isomorphic to the infinite dihedral group
and the region bounded by two adjacent edges of $\DD$
with an endpoint $r$ is a fundamental domain
for the action of $\RGP{r}$ on $\HH^2$.
Let $\RGPP{r}$ be the group generated by $\RGP{r}$ and $\RGP{\infty}$.
When $r\in \QQ - \ZZ$,
$\RGPP{r}$ is equal to the free product $\RGP{r}*\RGP{\infty}$,
having a fundamental domain shown in \fullref{fig.fd}.
When $r \in \ZZ \cup \{ \infty \}$,
we concretely describe $\hgr$
in \fullref{ex.r=infty} and \fullref{ex.r=integer}.
By using the fundamental domain
of the group $\RGPP{r}$,
we can give a practical algorithm to determine
whether a given rational number $\tilde r$
belongs to the $\RGPP{r}$--orbit of $\infty$ or $r$
(see \fullref{Sec:Continued fraction}).
In fact, \fullref{Prop:continued fraction}
characterizes such a rational number $\tilde r$
in terms of its continued fraction expansion.

\begin{figure}[ht!]
\begin{center}
\begin{picture}(350,165)
\put(0,69){\pc{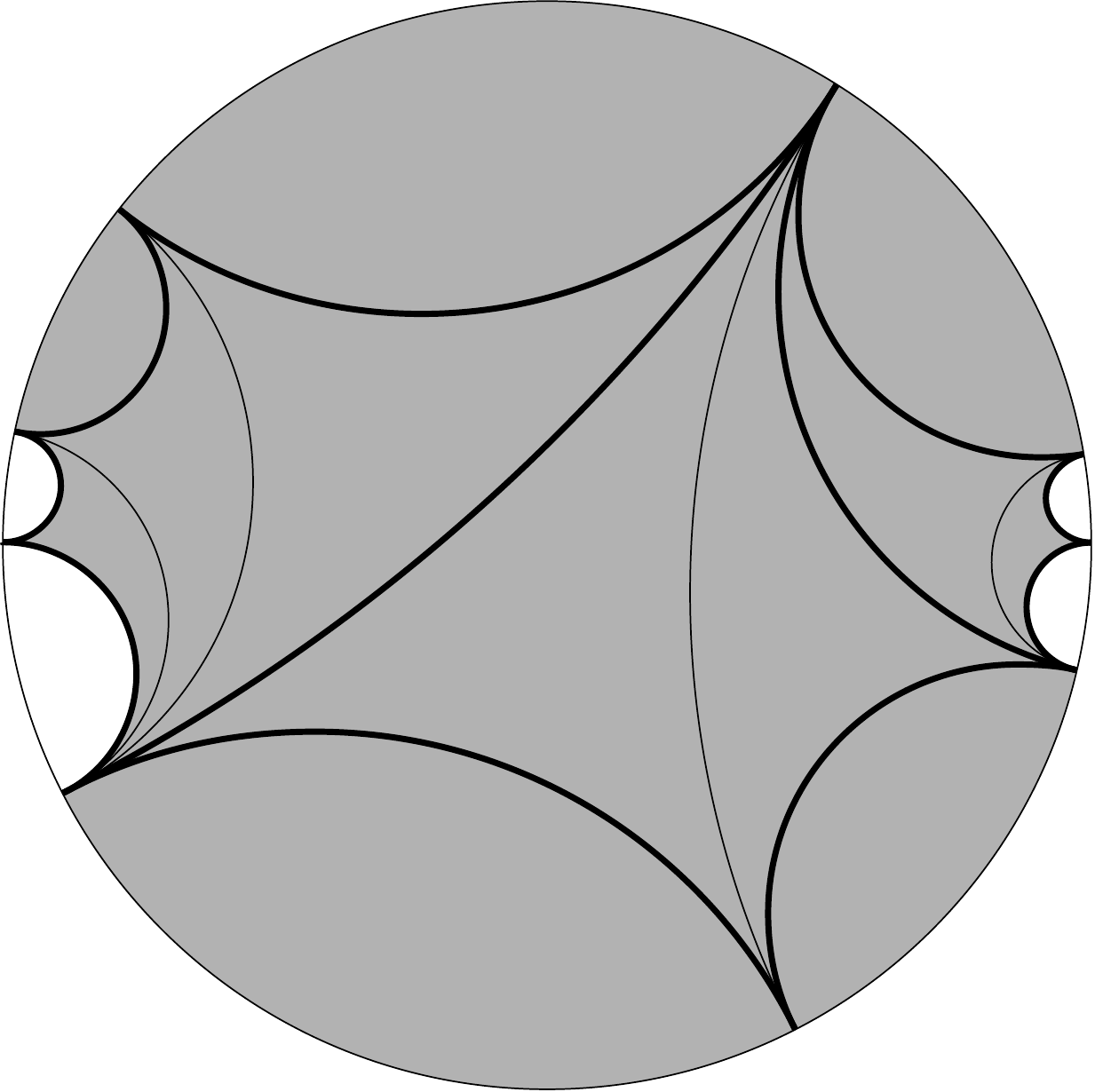}{0.4}}
\put(-8,72){\fns $\infty$}
\put(-2,87){\fns $1$}
\put(3,119){\fns $1/2$}
\put(110,142){\fns $1/3=[3]$}
\put(150,85){\fns $3/10$}
\put(150,71){\fns $5/17 = [3,2,2] = r$}
\put(148,53){\fns $2/7=[3,2]$}
\put(104,1){\fns $1/4$}
\put(4,37){\fns $0$}
\put(190,120){\pc{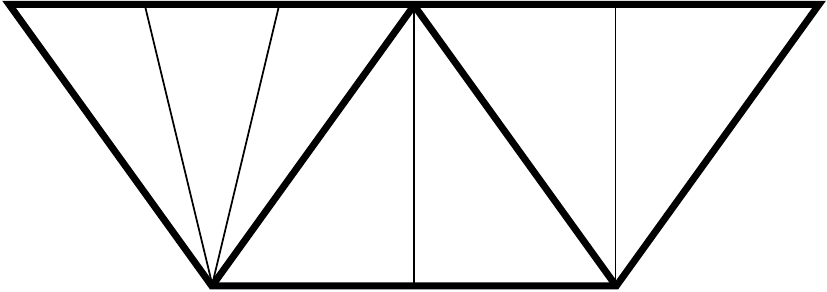}{0.5}}
\put(185,147){\small $\infty$}
\put(198,149){\large $\overbrace{\hspace*{4.5pc}}^{a_1=3 \ }$}
\put(250,158){\small $[3]$}
\put(257,149){$\overbrace{\hspace*{4.5pc}}^{\ a_3=2}$}
\put(318,147){\small $[3,2,2]=r$}
\put(216,101){\small $0$}
\put(227,102){$\underbrace{\hspace*{4.5pc}}_{a_2=2}$}
\put(288,101){\small $[3,2]$}
\put(195,28){$K(r)=$}
\put(224,27){\pc{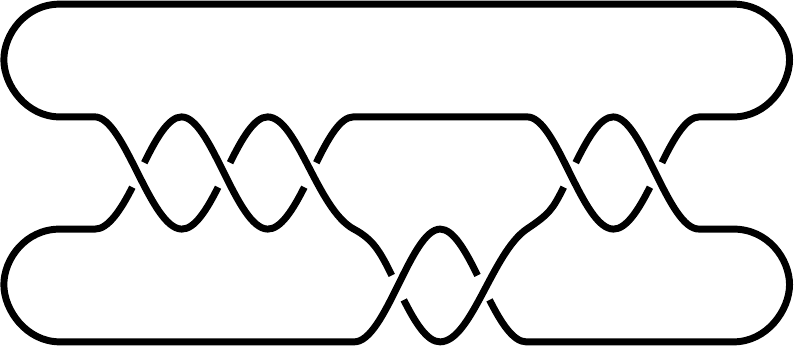}{0.6}}
\put(242,45){\small $3$ half-twists}
\put(305,45){\small $2$ half-twists}
\put(265,-7){\small $(-2)$ half-twists}
\end{picture}
\end{center}
\caption{
A fundamental domain of $\hgr$ in the modular tessellation
(the shaded domain),
the linearization of (the core part of) the fundamental domain,
and the $2$--bridge knot $K(r)$, for $r=5/17=[3,2,2]$}
\label{fig.fd}
\end{figure}
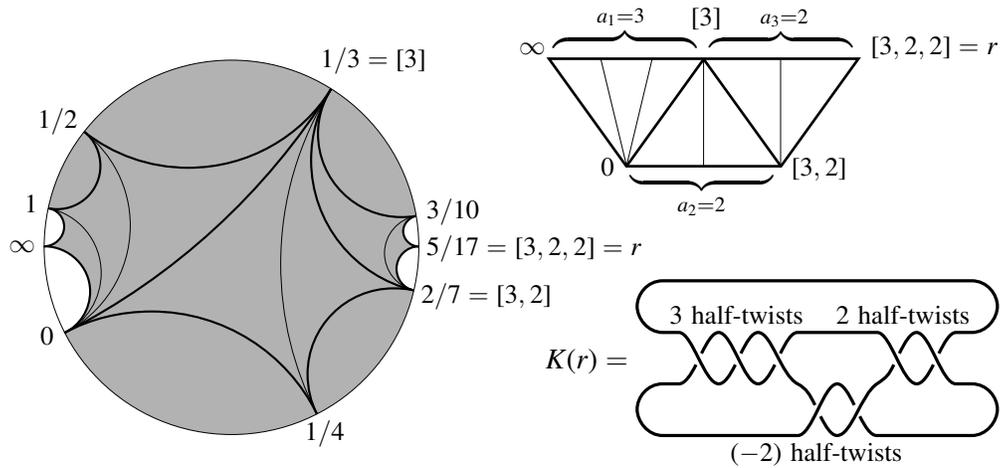

Now we study topological characterization of a link $\tilde{K}$
having an epimorphism $G(\tilde{K}) \to G(K)$
for a given link $K$.
We call a continuous map $f\co(S^3,\tilde{K}) \to (S^3,K)$
\textit{proper} if $\tilde{K} = f^{-1}(K)$.
Since a proper map induces a map between link complements,
it further induces a homomorphism $G(\tilde{K})\to G(K)$
preserving peripheral structure.
Conversely, any epimorphism $G(\tilde{K})\to G(K)$
for a nonsplit link $K$,
preserving peripheral structure,
is induced by some proper map $(S^3,\tilde{K}) \to (S^3,K)$,
because the complement of a nonsplit link is aspherical.
Thus, we can obtain $\tilde{K}$ as $f^{-1}(K)$
for a suitably chosen map $f\co S^3 \to S^3$;
in \fullref{q.const_tK} we propose a conjecture
to characterize $\tilde{K}$ from $K$ in this direction.

For $2$--bridge links, we have
the following theorem which implies that,
for each epimorphism $G(K(\tilde r)) \to G(K(r))$
in \fullref{Thm:epimorphism1},
we can topologically characterize $K(\tilde{r})$
as the preimage $f^{-1}(K(r))$
for some specific kind of a proper map $f$.
For the proof of the theorem,
see \fullref{sec.tr_Ktr}, \fullref{sec:construction1}
and \fullref{fig.plan}.

\begin{Theorem}
\label{Thm:epimorphism2}
If $\tilde r$ belongs to the $\RGPP{r}$--orbit of $r$ or $\infty$,
then there is a proper branched fold map
$f\co (S^3,K(\tilde r))\to(S^3,K(r))$
which induces an epimorphism $G(K(\tilde r)) \to G(K(r))$,
such that
its fold surface is the disjoint union of level $2$--spheres
and its branch curve is a link of index 2 disjoint to $K(\tilde{r})$,
each of whose components lie in a level $2$--sphere.
\end{Theorem}

Here, we explain the terminology in the theorem.
More detailed properties of the map $f$ are given
in \fullref{rem.f1} and \fullref{rem.f2}.

By a \textit{branched fold map}, we mean a map between $3$--manifolds
such that, for each point $p$ in the source manifold,
there exist local coordinates around $p$ and $f(p)$
such that $f$ is given by one of the following formulas
in the neighborhood of $p$:
\begin{align*}
& f(x_1,x_2,x_3) = (x_1,x_2,x_3) \\
& f(x_1,x_2,x_3) = (x_1^2,x_2,x_3) \\
& f(z,x_3) = (z^n,x_3) \qquad ( z = x_1 + x_2 \sqrt{-1} \, )
\end{align*}
When $p$ and $f(p)$ have such coordinates around them,
we call $p$
a \textit{regular point},
a \textit{fold point} or
a \textit{branch point} of \textit{index} $n$, accordingly.
The set of fold points forms a surface in the source manifold,
which we call the \textit{fold surface} of $f$.
The set of branch points forms a link in the source manifold,
which we call the \textit{branch curve} of $f$.
(If $f$ further allowed ``fold branch points''
which are defined by $f(x_1,z)=(\smash{x_1^2},z^2)$
for suitable local coordinates
where $z = x_2 + x_3 \sqrt{-1}$,
and if the index of every branch point is $2$,
$f$ is called a ``nice'' map in Honma \cite{Homma}.
It is shown \cite{Homma} that
any continuous map between $3$--manifolds
can be approximated by a ``nice'' map.)

A \textit{height function} for $K(r)$ is a function
$h:S^3\to [-1,1]$ such that
$h^{-1}(t)$ is a $2$--sphere intersecting $K(r)$ transversely in four 
points
or a disk intersecting $K(r)$ transversely in two points in its interior
according as $t\in (-1,1)$ or $\{\pm 1\}$.
We call the $2$--sphere $h^{-1}(t)$ with $t\in (-1,1)$ a \textit{level} 
$2$--sphere.

\fullref{Thm:epimorphism1} and \fullref{Thm:epimorphism2}
have applications to the character varieties of $2$--bridge links
and $\pi_1$--dominating maps among $3$--manifolds.
These are given in \fullref{sec.icv} and \fullref{Sec:pi_1-dominating 
maps}.

The paper is organized as follows;
see also \fullref{fig.plan} for a sketch plan
to prove \fullref{Thm:epimorphism1} and \fullref{Thm:epimorphism2}.
In \fullref{sec.ekg},
we quickly review known facts concerning
epimorphisms between knot groups,
in order to explain background and motivation
for the study of epimorphisms between knot groups.
In \fullref{Sec:tangle}, we review basic properties of $2$--bridge 
links.
In \fullref{Sec:algebraic proof},
we prove \fullref{Thm:epimorphism1},
constructing epimorphisms
$G(K(\tilde{r})) \to G(K(r))$.
In \fullref{sec.tr_Ktr},
we show that
if $\tilde{r}$ belongs to the $\hgr$--orbit of $r$ or $\infty$,
then $\tilde{r}$ has a continued fraction expansion
of a certain specific form
in \fullref{Prop:continued fraction},
and equivalently $K(\tilde{r})$ has a plat presentation
of a certain specific form
in \fullref{prop.hGr_Ktr}.
In \fullref{sec:construction1},
we give an explicit construction of the desired proper map
$(S^3,K(\tilde{r})) \to (S^3,K(r))$
under the setting of \fullref{prop.hGr_Ktr}
(see \fullref{thm.K_tildeK}).
This together with \fullref{prop.hGr_Ktr}
gives the proof of \fullref{Thm:epimorphism2}.
We also describe further properties of the map
in \fullref{sec:construction1}.
In \fullref{sec.icv} and \fullref{Sec:pi_1-dominating maps},
we show applications of
\fullref{Thm:epimorphism1} and \fullref{Thm:epimorphism2}
to the character varieties for $2$--bridge links
and $\pi_1$--dominating maps among $3$--manifolds.
In \fullref{sec:questions},
we propose some questions
related to \fullref{Thm:epimorphism1} and \fullref{Thm:epimorphism2}.

\begin{figure}[ht!]
\begin{center}
\begin{picture}(320,180)
\put(30,150){\framebox(220,25){\fns
$\tilde{r}$ belongs to the $\hgr$--orbit of $r$ or $\infty$}}
\put(40,140){\vector(-1,-4){27}}
\put(-25,110){\fns \fullref{Thm:epimorphism1}}
\put(60,95){\line(0,1){45}}
\put(60,95){\vector(3,-2){95}}
\put(50,55){\fns \fullref{Thm:epimorphism2}}
\put(140,140){\vector(0,-1){30}}
\put(145,120){\fns \fullref{Prop:continued fraction}}
\put(240,140){\vector(1,-2){15}}
\put(253,120){\fns \fullref{prop.hGr_Ktr}}
\put(120,70){\framebox(90,30){
\shortstack[l]{\fns $\tilde{r}$ is presented as \\
\fns in \fullref{Prop:continued fraction}. }}}
\put(215,80){$\Longleftrightarrow$}
\put(240,70){\framebox(100,30){
\shortstack[l]{\fns $K(\tilde{r})$ is presented as \\
\fns in \fullref{prop.hGr_Ktr}. }}}
\put(117,67){\framebox(226,36){\qquad}}
\put(237,62){\vector(-1,-2){15}}
\put(237,45){\fns \fullref{thm.K_tildeK}}
\put(145,-5){\framebox(160,30){
\shortstack[l]{\fns There exists a branched fold map \\
\fns $(S^3,K(\tilde{r})) \to (S^3,K(r))$.}}}
\put(120,10){$\Longleftarrow$}
\put(-25,-5){\framebox(140,30){
\shortstack[l]{\fns There exists an epimorphism \\
\fns $G(K(\tilde{r})) \to G(K(r))$.}}}
\end{picture}
\end{center}
\caption{\label{fig.plan}
Sketch plan to prove
\fullref{Thm:epimorphism1} and \fullref{Thm:epimorphism2}}
\end{figure}

\textbf{Personal history}\qua
This paper is actually an expanded version
of the unfinished joint paper \cite{Ohtsuki-Riley}
by the first and second authors
and the announcement \cite{Sakuma} by the third author.
As is explained in the introduction of Riley \cite{Riley3},
the first and second authors proved \fullref{thm.K_tildeK},
motivated by the study of reducibility of
the space of irreducible $SL(2,\CC)$ representations of
$2$--bridge knot groups,
and obtained (a variant) of \fullref{Cor:character variety1}.
On the other hand,
the last author discovered \fullref{Thm:epimorphism1} while doing joint
research \cite{ASWY2} with
H\,Akiyoshi, M\,Wada and Y\,Yamashita
on the geometry of $2$--bridge links.
This was made when he was visiting G\,Burde in the summer of 1997,
after learning several examples found by Burde and his student, 
F\,Opitz,
through computer experiments on representation spaces.
The first and third authors realized that
\fullref{Thm:epimorphism1} and \fullref{thm.K_tildeK}
are equivalent in the autumn of 1997,
and agreed to write a joint paper with the second author.
But, very sadly, the second author passed away on March 4, 2000,
before the joint paper was completed.
May Professor Robert Riley rest in peace.

\medskip{\bf Acknowledgments}\qua
The first author would like to thank Osamu Saeki
for helpful comments.
The third author would like to thank
Gerhard Burde and Felix Opitz for teaching him
of their experimental results on $2$--bridge knot groups.
He would also like to thank Michel Boileau,
Kazuhiro Ichihara and Alan Reid for useful information.
The first and the third authors would like to thank
Teruaki Kitano and Masaaki Suzuki
for informing the authors of their recent results.
They would also like to thank
Andrew Kricker, Daniel Moskovich and Kenneth Shackleton
for useful information.
Finally, they would like to thank the referee
for very careful reading and helpful comments.

\section{Epimorphisms between knot groups}
\label{sec.ekg}

In this section,
we summarize topics and known results
related to epimorphisms between knot groups,
in order to give background and motivation
to study epimorphisms between knot groups.

We have a partial order on the set of prime knots,
by setting $\tilde K\ge K$ if there is an
epimorphism $G(\tilde{K}) \to G(K)$.
A nontrivial part of the proof is to show that
$K_1 \ge K_2$ and $K_1 \le K_2$ imply $K_1 = K_2$,
which is shown from the following two facts.
The first one is that
we have a partial order on the set of knot groups of all knots;
its proof is due to the Hopfian property
(see, for example, Silver--Whitten \cite[Proposition 3.2]{Silver-Whitten}).
The second fact is that
prime knots are determined by their knot groups
(see, for example, Kawauchi \cite[Theorem 6.1.12]{Kawauchi_book}).

The existence and nonexistence of epimorphisms between knot groups
for some families of knots have been determined.
Gonzal\'ez-Ac\~una and Ram\'inez \cite{Gonzalez-Raminez1}
gave a certain topological characterization of those knots
whose knot groups have epimorphisms to torus knot groups,
in particular, they determined in \cite{Gonzalez-Raminez2} the
$2$--bridge knots whose knot groups have epimorphisms
to the $(2,p)$ torus knot group.
Kitano--Suzuki--Wada \cite{Kitano-Suzuki-Wada}
gave an effective criterion for the existence of an
epimorphism among two given knot groups, in terms of
the twisted Alexander polynomials,
extending the well-known criterion
that the Alexander polynomial of $\tilde{K}$ is divisible by that of $K$
if there is an epimorphism $G(\tilde{K})\to G(K)$.
By using the criterion,
Kitano--Suzuki \cite{Kitano-Suzuki}
gave a complete list of such pairs $(\tilde{K},K)$
among the prime knots with up to 10 crossings.

The finiteness of $K$
admitting an epimorphism $G(\tilde{K}) \to G(K)$
for a given $\tilde{K}$
was conjectured by Simon \cite[Problem 1.12]{Kirby},
and it was partially solved by
Boileau--Rubinstein--Wang \cite{Boileau-Rubinstein-Wang},
under the assumption that the epimorphisms are induced by
nonzero degree proper maps.

A systematic construction of epimorphisms between knot groups
is given by Kawauchi's imitation theory \cite{Kawauchi};
in fact,
his theory constructs an imitation $\tilde{K}$ of $K$
which shares various topological properties with $K$,
and, in particular,
there is an epimorphism between their knot groups.

{}From the viewpoint of maps between ambient spaces,
any epimorphism $G(\tilde{K})\to G(K)$
for a nonsplit link $K$,
preserving peripheral structure,
is induced by some proper map $f\co (S^3,\tilde{K}) \to (S^3,K)$,
as mentioned in the introduction.
The index of the image $f_*(G(\tilde{K}))$ in $G(K)$
is a divisor of the degree of $f$  (see Hempel \cite[Lemma 15.2]{Hemple}).
In particular, if $f$ is of degree $1$,
then $f_*$ induces an epimorphism between the knot groups.
Thus the problem of epimorphisms between knot groups
is related to the study of proper maps between ambient spaces,
more generally, maps between $3$--manifolds.
This direction has been extensively studied
in various literatures
(see Boileau--Wang \cite{Boileau-Wang}, Boileau--Rubinstein--Wang \cite{Boileau-Rubinstein-Wang}, Kawauchi \cite{Kawauchi},
Reid--Wang--Zhou \cite{Reid-Wang-Zhou},  Silver--Whitten \cite{Silver-Whitten}, Soma \cite{Soma1,Soma2}, Wang \cite{Wang}
and references therein).

\section[Rational tangles and 2-bridge links]{Rational tangles and $2$--bridge links}
\label{Sec:tangle}

In this section, we recall basic definitions and
facts concerning the $2$--bridge knots and links.

Consider the discrete group, $H$, of isometries
of the Euclidean plane $\RR^2$
generated by the $\pi$--rotations around
the points in the lattice $\ZZ^2$.
Set $\Conways=(\RR^2,\ZZ^2)/H$
and call it the \textit{Conway sphere}.
Then $\Conway$ is homeomorphic to the $2$--sphere,
and $\PP$ consists of four points in $\Conway$.
We also call $\Conway$ the Conway sphere.
Let $\PConway:=\Conway-\PP$ be the complementary
$4$--times punctured sphere.
For each $r \in \QQQ:=\QQ\cup\{\infty\}$,
let $\alpha_r$ be the simple loop in $\PConway$
obtained as the projection of the line in $\RR^2-\ZZ^2$
of slope $r$.
Then $\alpha_r$ is \textit{essential} in $\PConway$,
ie, it does not bound a disk in $\PConway$
and is not homotopic to a loop around a puncture.
Conversely, any essential simple loop in $\PConway$
is isotopic to $\alpha_r$ for
a unique $r\in\QQQ$.
Then $r$ is called the \textit{slope} of the simple loop.
Similarly, any simple arc $\delta$
in $\Conway$ joining two
different points in $\PP$ such that
$\delta\cap \PP=\partial\delta$
is isotopic to the image of a line in $\RR^2$
of some slope $r\in\QQQ$
which intersects $\ZZ^2$.
We call $r$ the \textit{slope} of $\delta$.

A \textit{trivial tangle} is a pair $(B^3,t)$,
where $B^3$ is a 3-ball and $t$ is a union of two
arcs properly embedded in $B^3$
which is parallel to a union of two
mutually disjoint arcs in $\partial B^3$.
Let $\tau$ be the simple unknotted arc in $B^3$
joining  the two components of $t$
as illustrated in \fullref{fig.trivial-tangle}.
We call it the \textit{core tunnel} of the trivial tangle.
Pick a base point $x_0$ in $\interiorop \tau$,
and let $(\mu_1,\mu_2)$ be the generating pair
of the fundamental group $\pi_1(B^3-t,x_0)$
each of which is represented by
a based loop consisting of
a small peripheral simple loop
around a component of $t$ and
a subarc of $\tau$ joining the circle to $x$.
For any base point $x\in B^3-t$,
the generating pair of $\pi_1(B^3-t,x)$
corresponding to the generating pair $(\mu_1,\mu_2)$
of $\pi_1(B^3-t,x_0)$ via a path joining $x$ to $x_0$
is denoted by the same symbol.
The pair $(\mu_1,\mu_2)$ is unique up to
(i) reversal of the order,
(ii) replacement of one of the members with its inverse,
and (iii) simultaneous conjugation.
We call the equivalence class of $(\mu_1,\mu_2)$
the \textit{meridian pair} of the fundamental
group $\pi_1(B^3-t)$.

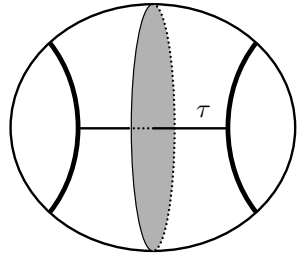
\begin{figure}[ht!]
$$
\pc{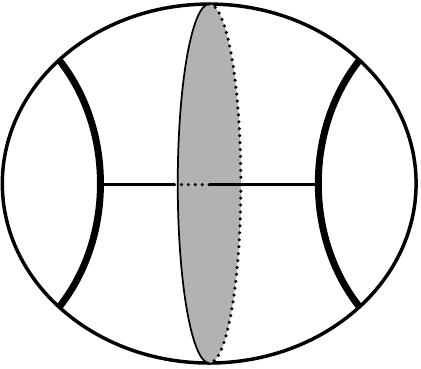}{0.9}
\begin{picture}(0,0)
\put(-45,9){\small$\tau$}
\end{picture}
$$
\caption{\label{fig.trivial-tangle}
A trivial tangle}
\end{figure}

By a \textit{rational tangle},
we mean a trivial tangle $(B^3,t)$
which is endowed with a homeomorphism from
$\partial(B^3,t)$ to $\Conways$.
Through the homeomorphism we identify
the boundary of a rational tangle with the Conway sphere.
Thus the slope of an essential simple loop in
$\partial B^3-t$ is defined.
We define
the \textit{slope} of a rational tangle
to be the slope of
an essential loop on $\partial B^3 -t$ which bounds a disk in $B^3$
separating the components of $t$.
(Such a loop is unique up to isotopy
on $\partial B^3 -t$ and is called a \textit{meridian}
of the rational tangle.)
We denote a rational tangle of slope $r$ by
$\rtangle{r}$.
By van Kampen's theorem, the fundamental group
$\pi_1(B^3-t(r))$ is identified
with the quotient
$\pi_1(\PConway)/\llangle\alpha_r \rrangle$,
where $\llangle\alpha_r \rrangle$ denotes the normal
closure.

For each $r\in \QQQ$,
the \textit{$2$--bridge link $K(r)$ of slope $r$}
is defined to be the sum of the rational tangles of slopes
$\infty$ and $r$, namely,
$(S^3,K(r))$ is
obtained from $\rtangle{\infty}$ and
$\rtangle{r}$
by identifying their boundaries through the
identity map on the Conway sphere
$\Conways$. (Recall that the boundaries of
rational tangles are identified with the Conway sphere.)
$K(r)$ has one or two components according as
the denominator of $r$ is odd or even.
We call $\rtangle{\infty}$
and $\rtangle{r}$, respectively,
the \textit{upper tangle} and \textit{lower tangle}
of the $2$--bridge link.
The image of the core tunnels for $\rtangle{\infty}$
and $\rtangle{r}$ are called the
\textit{upper tunnels} and \textit{lower tunnel}
for the $2$--bridge link.

We describe a plat presentation of $K(r)$, as follows.
Choose a continued fraction expansion of $r$,
$$
r=[a_1,a_2,\cdots,a_m] .
$$
When $m$ is odd, we have a presentation,
$$
r = B \cdot \infty
\quad \mbox{ where }
B =
\mbox{\small $\begin{pmatrix} 1 & 0 \\ a_1 & 1 \end{pmatrix}$}
\mbox{\small $\begin{pmatrix} 1 & a_2 \\ 0 & 1 \end{pmatrix}$}
\mbox{\small $\begin{pmatrix} 1 & 0 \\ a_3 & 1 \end{pmatrix}$} \cdots
\mbox{\small $\begin{pmatrix} 1 & 0 \\ a_m & 1 \end{pmatrix}$} ,
$$
and $B$ acts on $\QQQ$
by the linear fractional transformation.
Then $K(r)$ has the following presentation,
where the boxed ``$a_i$'' implies $a_i$ half-twists.
\eject
$$
\begin{picture}(300,120)
\put(125,102){\fns monodromy $=B$}
\put(257,113){\vector(-1,0){180}}
\put(56,110){$T^2$}
\put(61,105){\vector(0,-1){35}}
\put(28,85){\shortstack[l]{\fns double \\ \fns cover}}
\put(264,110){$T^2$}
\put(267,105){\vector(0,-1){30}}
\put(270,85){\shortstack[l]{\fns double \\ \fns cover}}
\put(0,37){$K(r) = $}
\put(40,35){\pc{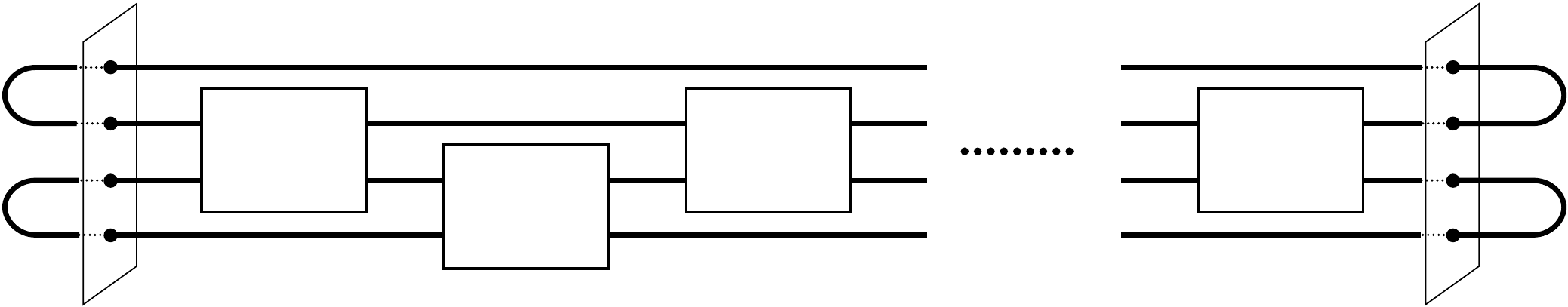}{0.4}}
\put(84,39){\small $a_1$}
\put(116,30){\small $-a_2$}
\put(158,39){\small $a_3$}
\put(234,39){\small $a_m$}
\put(70,65){$\overbrace{\hspace*{18pc}}^{t(r)}$}
\put(43,17){$\underbrace{\hspace*{1.5pc}}_{\overline{t(\infty)}}$}
\put(270,17){$\underbrace{\hspace*{1.5pc}}_{t(\infty)}$}
\end{picture}
$$
Similarly, when $m$ is even:
$$
r =
\mbox{\small $\begin{pmatrix} 1 & 0 \\ a_1 & 1 \end{pmatrix}$}
\mbox{\small $\begin{pmatrix} 1 & a_2 \\ 0 & 1 \end{pmatrix}$}
\mbox{\small $\begin{pmatrix} 1 & 0 \\ a_3 & 1 \end{pmatrix}$} \cdots
\mbox{\small $\begin{pmatrix} 1 & a_m \\ 0 & 1 \end{pmatrix}$}
\cdot 0 ,
$$
$$
\begin{picture}(300,80)
\put(0,37){$K(r) = $}
\put(40,35){\pc{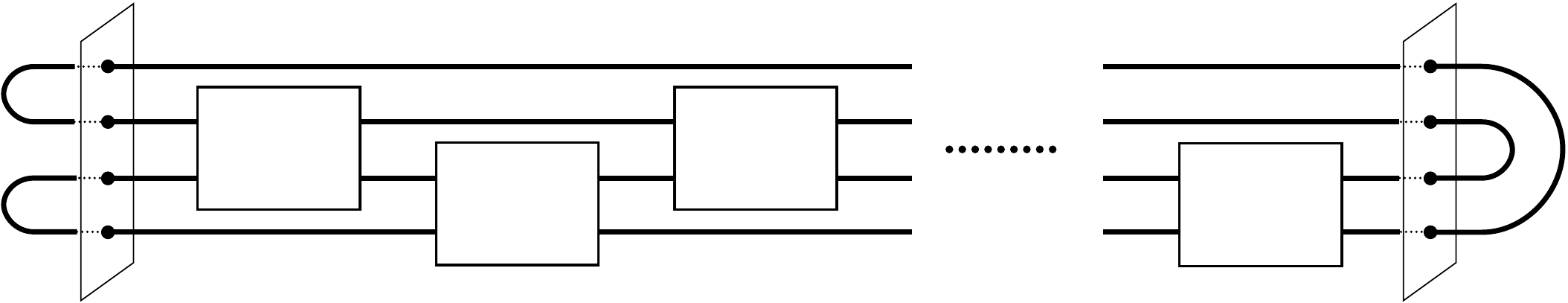}{0.4}}
\put(84,39){\small $a_1$}
\put(116,30){\small $-a_2$}
\put(158,39){\small $a_3$}
\put(230,30){\small $-a_m$}
\put(70,65){$\overbrace{\hspace*{18pc}}^{t(r)}$}
\put(43,17){$\underbrace{\hspace*{1.5pc}}_{\overline{t(\infty)}}$}
\put(270,17){$\underbrace{\hspace*{1.5pc}}_{t(0)}$}
\end{picture}
$$
We recall Schubert's classification \cite{Schubert}
of the $2$--bridge links
(cf \cite{Burde-Zieschang}).

\begin{Theorem}[Schubert]
\label{Thm:Schubert}
Two $2$--bridge links $K(q/p)$ and $K(q'/p')$
are equivalent,
if and only if the following conditions hold.
\begin{enumerate}
\item
$p=p'$.
\item
Either $q\equiv \pm q'\pmod p$
or $qq'\equiv \pm 1\pmod p$.
\end{enumerate}
Moreover, if the above conditions are satisfied, 
there is a
homeomorphism
$$f\co (S^3,K(q/p))\to (S^3,K(q'/p'))$$
which satisfies the following conditions.
\begin{enumerate}
\item
If $q\equiv q'\pmod p$
or $qq'\equiv 1\pmod p$,
then $f$ preserves the orientation of $S^3$.
Otherwise, $f$ reverses the orientation of $S^3$.
\item
If $q\equiv \pm q'\pmod p$,
then $f$ maps the upper tangle of $K(q/p)$
to that of $K(q'/p')$.
If $qq'\equiv 1\pmod p$,
then $f$ maps the upper tangle of $K(q/p)$
to the lower tangle of $K(q'/p')$.
\end{enumerate}
\end{Theorem}

By van Kampen's theorem,
the link group
$G(K(r))=\pi_1(S^3-K(r))$
of $K(r)$ is identified with
$\pi_1(\PConway)/
\llangle\alpha_{\infty},\alpha_r\rrangle$.
We call the image in the link group
of the meridian pair
of the fundamental group $\pi_1(B^3-t(\infty))$
(resp.\ $\pi_1(B^3-t(r))$
the \textit{upper meridian pair}
(resp.\ \textit{lower meridian pair}).
The link group is regarded as the quotient
of the rank 2 free group,
$\pi_1(B^3-t(\infty))\cong\pi_1(\PConway)/
\llangle\alpha_{\infty}\rrangle$,
by the normal closure of $\alpha_{r}$.
This gives a one-relator presentation of the
link group, and is actually equivalent to the
upper presentation \cite{Crowell-Fox}.
Similarly,
the link group is regarded as the quotient
of the rank 2 free group
$\pi_1(B^3-t(r))\cong\pi_1(\PConway)/
\llangle\alpha_{r}\rrangle$
by the normal closure of $\alpha_{\infty}$,
which in turn
gives the
lower presentation of the link group.
These facts play an important role in the
next section.

\section[Constructing an epimorphism from G(K(tilde r)) to G(K(r))]{Constructing an epimorphism $G(K(\tilde{r})) \to G(K(r))$}
\label{Sec:algebraic proof}

In this section, we prove \fullref{Thm:epimorphism1},
which states the existence of an epimorphism $G(K(\tilde{r})) \to 
G(K(r))$.
Before proving the theorem,
we explain special cases of
the theorem for some simple values of $r$.

\begin{Example}
\label{ex.r=infty}\rm
If $r=\infty$, then $K(r)$ is a trivial $2$--component link.
Further, $\RGPP{r}=\RGP{r}=\RGP{\infty}$.
Thus the region bounded by the edges $\langle\infty,0\rangle$
and $\langle\infty,1\rangle$ is a fundamental domain
for the action of $\RGPP{r}$ on $\HH^2$.
Hence, the assumption of \fullref{Thm:epimorphism1}
is satisfied if and only if $\tilde r=\infty$.
This reflects the fact that
a link is trivial if and only if its link group is a free group.
\end{Example}

\begin{Example}
\label{ex.r=integer}\rm
If $r\in\ZZ$, then $K(r)$ is a trivial knot.
Further, $\RGPP{r}$ is equal to the group
generated by the reflections in the edges of any of $\DD$.
In particular, any ideal triangle of $\DD$
is a fundamental domain for the action of $\RGPP{r}$ on $\HH^2$.
Hence, $\RGPP{r}$ acts transitively on $\QQQ$
and every $\tilde r\in\QQQ$ satisfies the assumption
of \fullref{Thm:epimorphism1}.
This reflects the fact that
there is an epimorphism from the link group
of an arbitrary link $L$
to $\ZZ$, the knot group of the trivial knot,
sending meridians to meridians.
\end{Example}

\begin{Example}
\label{ex.r=half_integer}\rm
If $r\equiv 1/2 \pmod \ZZ$, then $K(r)$ is a Hopf link.
Further, $\tilde r=q/p$ satisfies the assumption
of \fullref{Thm:epimorphism1}
if and only if $p$ is even, ie,
$K(\tilde r)$ is a $2$--component link.
This reflects the fact that
the link group of an arbitrary $2$--component link
has an epimorphism to the link group, $\ZZ\oplus\ZZ$, of the Hopf link.
\end{Example}

The proof of \fullref{Thm:epimorphism1}
is based on the following simple observation.
\eject
\begin{Lemma}
\label{Lem:Komori-Series}
For each rational tangle $\rtangle{r}$,
the following holds.
\begin{enumerate}
\item
For each $s\in\QQQ$,
the simple loop $\alpha_s$ is nullhomotopic
in $B^3-t(r)$
if and only if $s=r$.
\item
Let $s$ and $s'$ be elements of $\QQQ$
which belongs to the same $\RGP{r}$--orbit.
The the simple loops
$\alpha_s$ and $\alpha_{s'}$ are
homotopic in
$B^3-t(r)$.
\end{enumerate}
\end{Lemma}

\begin{proof}
The linear action of $SL(2,\ZZ)$ on
$\RR^2$ descends to an action on
$\Conways$, and the assertions in this lemma
are invariant by the action.
Thus we may assume $r=\infty$.

(1)\qua Let $\gamma_1$ and $\gamma_2$ be arcs
in $\partial \rtangle{r}$ of slope $\infty$,
namely $(\gamma_1\cup \gamma_2)\cap \partial t(\infty)=\partial t(\infty)$
and $\gamma_1\cup \gamma_2$ is parallel to $t(\infty)$
in $B^3$.
Then $\pi_1(B^3-t(\infty))$
is the free group of rank $2$ generated by
the meridian pair $\{\mu_1,\mu_2\}$,
and the cyclic word in $\{\mu_1,\mu_2\}$
obtained by reading the intersection
of the loop $\alpha_s$ with $\gamma_1\cup \gamma_2$
represent the free homotopy class of $\alpha_s$.
(After a suitable choice of orientation,
a positive intersection with $\gamma_i$
corresponds to $\mu_i$.
If $s\ne 0$, then $\alpha_s$ intersects 
$\gamma_1$ and $\gamma_2$ alternatively,
and hence the corresponding word is a reduce word.
Thus $\alpha_s$ is not nullhomotopic
in $B^3-t(r)$ if $s\ne \infty$.
Since the converse is obvious,
we obtain the desired result.

(2)\qua Let $A$ be the reflection of $\DD$
in the edge $\langle 0,\infty\rangle$,
and let $B$ be the
parabolic transformation of $\DD$ around the
vertex $\infty$ by $2$ units.
Then their actions on $\QQQ$ is given by
$A(s)=-s$ and $B(s)=s+2$.
Since $A$ and $B$ generates the
group $\RGP{\infty}$,
we have only to show that
the simple loop $\alpha_s$ on $\partial B^3-t(\infty)$
is homotopic to the simple loops of
slopes $-s$ and $s+2$
in $B^3-t(\infty)$

We first show that $\alpha_s$ is homotopic to
$\alpha_{-s}$ in $B^3-t(\infty)$.
Let $\RX$ be the orientation-reversing involution of
$\Conways$ induced by the reflection
$(x,y)\mapsto (x,-y+1)$ on $\RR^2$.
The fixed point set of $\RX$ is the simple loop of
slope $0$ which is obtained as the image of
the line $\RR\times\{1/2\}$.
The quotient space $\PConway/\RX$ is
homeomorphic to a twice punctured disk,
which we denote by $R$.
The projection $\PConway\to R$
extends to a continuous map
$B^3-t(\infty)\to R$,
which is a homotopy equivalence.
Then the two loops $\alpha_s$ and $\alpha_{-s}$
project to the same loop in $R$
and hence they must be homotopic in
$B^3-t(\infty)$.

Next, we show that
show that $\alpha_s$ is homotopic to
$\alpha_{s+2}$ in $B^3-t(\infty)$.
To this end, consider the Dehn twist
of $B^3-t(\infty)$ along the \lq\lq meridian disk'',
ie, the disk in $B^3-t(\infty)$
bounded by the simple loop $\alpha_{\infty}$.
Then it is homotopic to the identity map,
and maps $\alpha_s$ to $\alpha_{s+2}$.
Hence $\alpha_s$ is homotopic to
$\alpha_{s+2}$ in $B^3-t(\infty)$.
\end{proof}

\begin{Remark}\rm
\fullref{Lem:Komori-Series} is nothing other than
a reformulation of (a part of)
Theorem 1.2 of Komori and Series
\cite{Komori-Series},
which in turn is a correction
of Remark 2.5
of Keen--Series \cite{Keen-Series}.
However, we presented a topological proof,
for the sake of completeness.
Their theorem actually implies that the converse to
the second assertion of the lemma holds.
Namely, two simple loops
$\alpha_s$ and $\alpha_{s'}$ are
homotopic in
$B^3-t(r)$ if and only if they belong to the same
orbit of $\RGP{r}$.
This is also proved by using the fact that
$\pi_1(B^3-t(r))$ is the free group of rank $2$
generated by the meridian pair.
\end{Remark}

\fullref{Lem:Komori-Series} implies the following consequence
for $2$--bridge knots.

\begin{Proposition}\label{Prop:Knotgroup}
For every $2$--bridge knot $K(r)$, the following holds.
If two elements $s$ and $s'$ of $\QQQ$
lie in the same $\RGPP{r}$--orbit,
then $\alpha_s$ and $\alpha_{s'}$ are homotopic in
$S^3-K(r)$.
\end{Proposition}

\begin{proof}
Since $\RGPP{r}$ is generated by
$\RGP{\infty}$ and $\RGP{r}$,
we have only to show the assertion
when $s'=A(s)$ for some $A$ in
$\RGP{\infty}$ or $\RGP{r}$.
If $A\in \RGP{\infty}$,
then $\alpha_s$ and $\alpha_{s'}$ are homotopic in
$B^3-t(\infty)$ by \fullref{Lem:Komori-Series}.
Since $G(K(r))$ is a quotient of
$\pi_1(B^3-t(\infty))$, this implies that
$\alpha_s$ and $\alpha_{s'}$ are homotopic in
$S^3-K(r)$.
Similarly, if
$A\in \RGP{r}$,
then $\alpha_s$ and $\alpha_{s'}$ are homotopic in
$B^3-t(r)$ by \fullref{Lem:Komori-Series}.
Since $G(K(r))$ is a quotient of
$\pi_1(B^3-t(r))$, this also implies that
$\alpha_s$ and $\alpha_{s'}$ are homotopic in
$S^3-K(r)$.
This completes the proof of the proposition.
\end{proof}

\begin{Corollary}\label{Cor:nullhomotopic}
If $s$ belongs to the orbit of $\infty$ or $r$
by $\RGPP{r}$,
then $\alpha_s$ is nullhomotopic in $S^3-K(r)$.
\end{Corollary}

\begin{proof}
The loops $\alpha_{\infty}$ and $\alpha_r$ are
nullhomotopic in $B^3-t(\infty)$ and $B^3-t(r)$,
respectively.
Hence both of them are nullhomotopic in $S^3-K(r)$.
Thus we obtain the corollary by
\fullref{Prop:Knotgroup}.
\end{proof}

We shall discuss more about the corollary in \fullref{sec:questions}.

\begin{proof}
[Proof of \fullref{Thm:epimorphism1}]
Suppose $\tilde r$ belongs to the
orbit of $r$ or $\infty$ by $\RGPP{r}$.
Then $\alpha_{\tilde r}$
is nullhomotopic in
$G(K(r))=\pi_1(\PConway)/
\llangle\alpha_{\infty},\alpha_r\rrangle$.
Thus the normal closure
$\llangle\alpha_{\infty},\alpha_{\tilde r}\rrangle$
in $\pi_1(\PConway)$
is contained in
$\llangle\alpha_{\infty},\alpha_r\rrangle$.
Hence the identity map on $\pi_1(\PConway)$
induces an epimorphism
from
$G(K(\tilde r))=\pi_1(\PConway)/
\llangle\alpha_{\infty},\alpha_{\tilde r}\rrangle$
to
$G(K(r))=\pi_1(\PConway)/
\llangle\alpha_{\infty},\alpha_r\rrangle$.
It is obvious that the epimorphism
sends the upper meridian pair of
$G(K(\tilde r))$
to that of
$G(K(r))$.
This completes the proof of
\fullref{Thm:epimorphism1}.
\end{proof}

\section[Continued fraction expansion of tilde r in hat(Gamma)-r-orbits]{Continued fraction expansion of $\tilde{r}$ in 
$\hgr$--orbits}
\label{sec.tr_Ktr}

In this section,
we explain what $\tilde{r}$ and $K(\tilde{r})$ look like
when $\tilde{r}$ belongs to the $\hgr$--orbit of $r$ or $\infty$,
in \fullref{Prop:continued fraction} and \fullref{prop.hGr_Ktr}.
These propositions are substantially equivalent.

For the continued fraction expansion
$r=[a_1,a_2, \cdots,a_{m}]$,
let $\cfr$, $\cfr^{-1}$, $\epsilon\cfr$
and $\epsilon\cfr^{-1}$,
with $\epsilon\in\{-,+\}$, be the finite sequences
defined as follows:
\begin{align*}
\cfr &= (a_1, a_2,\cdots, a_m), \quad &
\cfr^{-1} &=(a_m,a_{m-1},\cdots,a_1),\\
\epsilon\cfr &=(\epsilon a_1,\epsilon a_2,\cdots,
\epsilon a_m), \quad &
\epsilon \cfr^{-1} &=(\epsilon a_m,\epsilon
a_{m-1},\cdots,
\epsilon a_1).
\end{align*}
Then we have the following proposition,
which is proved in \fullref{Sec:Continued fraction}.

\begin{Proposition}
\label{Prop:continued fraction}
Let $r$ be as above.
Then a rational number $\tilde r$
belongs to the orbit
of $r$ or $\infty$ by $\RGPP{r}$
if and only if
$\tilde r$ has the following continued
fraction expansion:
\[
\tilde r=
2c+[\epsilon_1\cfr,2c_1,\epsilon_2\cfr^{-1},2c_2,\epsilon_3\cfr,
\cdots, 2c_{n-1},\epsilon_n \cfr^{(-1)^{n-1}}]
\]
for some positive integer $n$,
$c\in\ZZ$,
$(\epsilon_1,\epsilon_2,\cdots,\epsilon_n)
\in \{-,+\}^n$
and
$(c_1,c_2,\cdots,c_{n-1})\in\ZZ^{n-1}$.
\end{Proposition}

The following proposition is a variation of \fullref{Prop:continued
fraction}, written in topological words,
which is proved in \fullref{sec.present_Ktr}.

\begin{Proposition}
\label{prop.hGr_Ktr}
We present the $2$--bridge link $K(r)$ by the plat closure
$$
\raisebox{12pt}[0pt][0pt]{$K(r) = \, $}
\begin{picture}(100,35)
\put(0,10){\pc{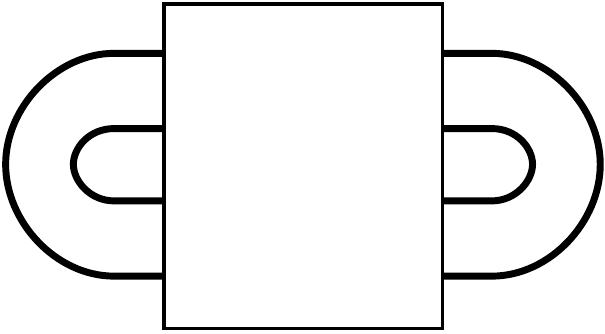}{0.4}}
\put(35,12){\small $b$}
\end{picture}
$$
for some $4$--braid $b$.
Then $\tilde{r}$ belongs to the $\hgr$--orbit of $\infty$ or $r$
if and only if
$K(\tilde{r})$ is presented by
$$
\raisebox{12pt}[0pt][0pt]{$K(\tilde{r}) = \, $}
\begin{picture}(330,35)
\put(0,10){\pc{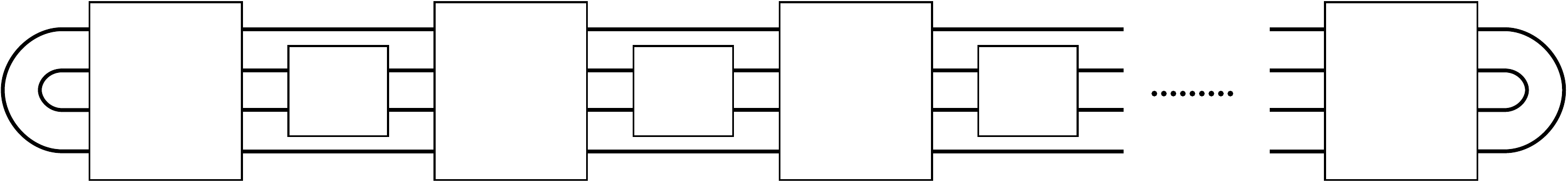}{0.4}}
\put(33,13){\small$b_\pm$}
\put(68,13){\small $2c_1$}
\put(103,13){\small$b_\pm^{-1}$}
\put(140,13){\small $2c_2$}
\put(178,13){\small$b_\pm$}
\put(213,13){\small $2c_3$}
\put(292,13){\small$\smash{b_\pm^{\pm1}}$}
\end{picture}
$$
for some signs of $b_\pm$ and $\smash{b_\pm^{-1}}$ and for some integers $c_i$,
where
a boxed ``$2c_i$'' implies $2c_i$ half-twists,
and $\smash{b_\pm^{\pm1}}$ are the braids obtained from $b$ by mirror images
as shown in the forthcoming \fullref{thm.K_tildeK}.
\end{Proposition}

\subsection[Continued fraction expansions of tilde r and r]{Continued fraction expansions of $\tilde{r}$ and $r$}
\label{Sec:Continued fraction}

In this section,
we prove \fullref{Prop:continued fraction}.
The proof is based on
the correspondence between
the modular tessellation and continued fraction expansions
(see Hatcher--Thurston \cite[p 229 Remark]{Hatcher-Thurston} for this correspondence).

We first recall the correspondence between
continued fraction expansions and edge-paths
in the modular diagram $\DD$.
For the continued fraction
expansion $r=[a_1,a_2,\cdots,a_m]$,
set $r_{-1}=\infty$, $r_0=0$
and $r_j=[a_1,a_2,\cdots, a_j]$ ($1\le j\le m$).
Then $(r_{-1},r_0,r_1,\cdots,r_m)$
determines an edge-path in $\DD$,
ie, $\langle r_j, r_{j+1}\rangle$ is
an edge of $\DD$ for each $j$ ($-1\le j\le m-1$).
Moreover, each component $a_j$ of the continued fraction
is read from the edge-path by the following rule:
The vertex $r_{j+1}$ is the image of $r_{j-1}$
by the parabolic transformation of $\DD$,
centered on the vertex $r_j$,
by $(-1)^j a_j$ units in the
clockwise direction.
(Thus the transformation is conjugate to
$$\begin{pmatrix} 1&(-1)^{j-1} a_j \\ 0&1 \end{pmatrix}$$
in $PSL(2,\ZZ)$.)
See \fullref{fig.cf}.

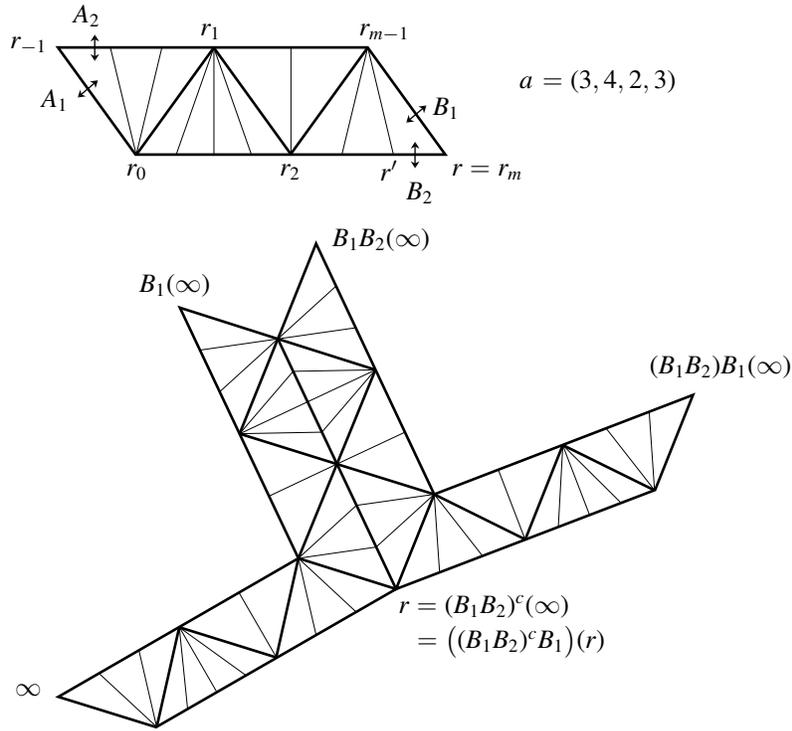
\begin{figure}[ht!]
\begin{center}
\begin{picture}(270,260)
\put(0,220){\pc{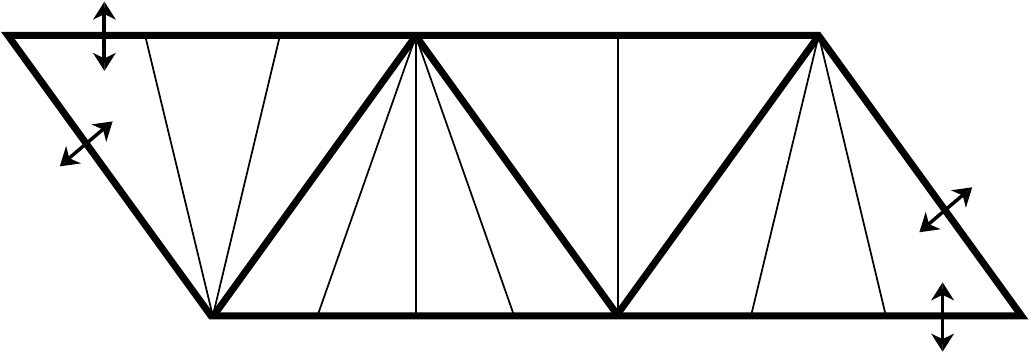}{0.5}}
\put(-12,245){\fns $r_{-1}$}
\put(12,255){\fns $A_2$}
\put(0,223){\fns $A_1$}
\put(60,250){\fns $r_1$}
\put(120,250){\fns $r_{m-1}$}
\put(180,230){\fns $\cfr = (3,4,2,3)$}
\put(148,220){\fns $B_1$}
\put(32,197){\fns $r_0$}
\put(90,197){\fns $r_2$}
\put(128,195){\fns $r'$}
\put(138,188){\fns $B_2$}
\put(155,197){\fns $r = r_m$}
\put(0,75){\pc{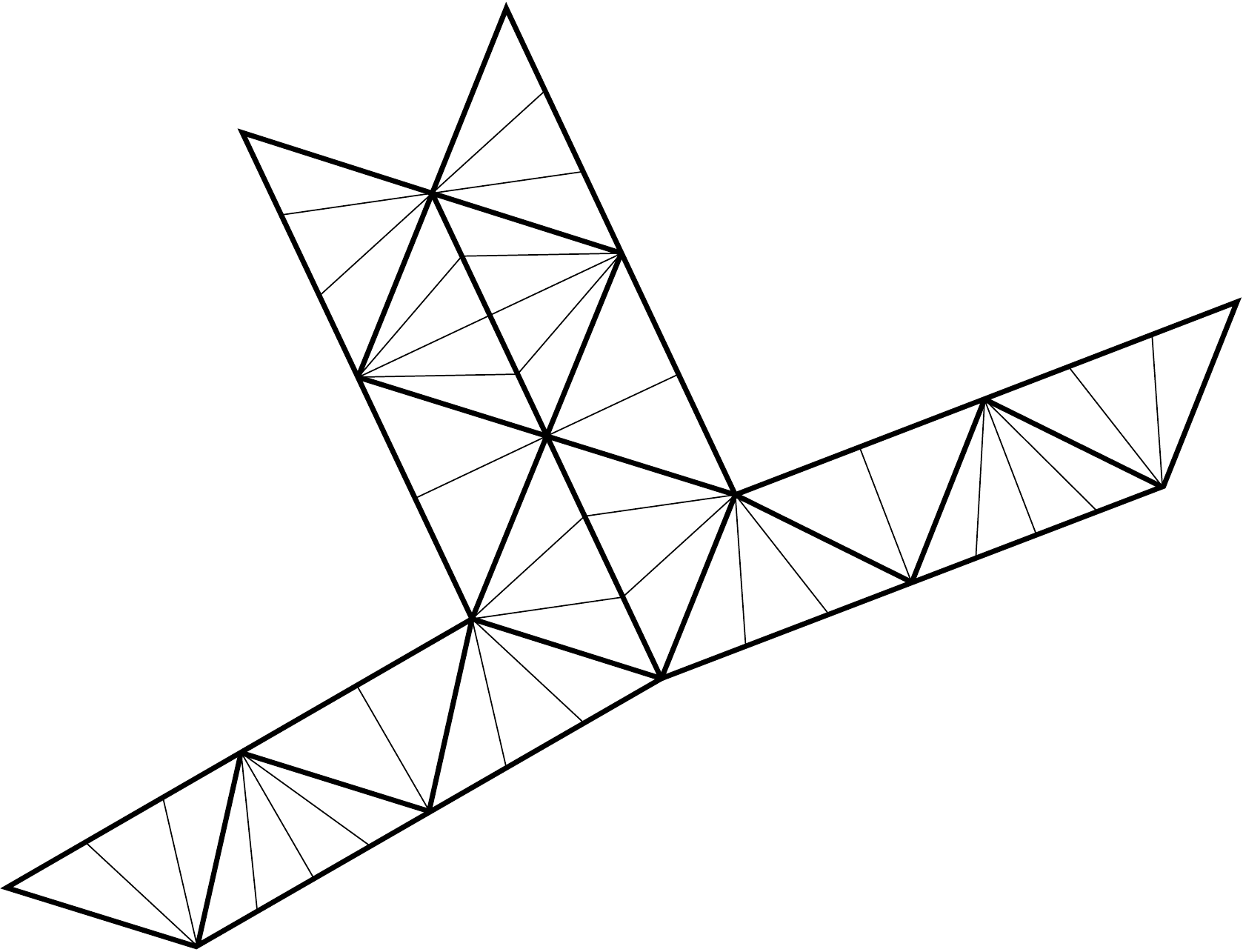}{0.5}}
\put(-10,0){\fns $\infty$}
\put(37,153){\fns $B_1 ( \infty )$}
\put(110,170){\fns $B_1 B_2 ( \infty )$}
\put(230,122){\fns $(B_1 B_2) B_1 ( \infty )$}
\put(135,18){\shortstack[l]{\fns $r = (B_1 B_2)^c ( \infty )$
\\ \fns \hspace*{4.5pt} $= \big( (B_1 B_2)^c B_1 \big)(r)$}}
\end{picture}
\end{center}
\caption{\label{fig.cf}
Continued fractions}
\end{figure}

Conversely, any edge-path
$(s_{-1},s_0,s_1,\cdots,s_m)$ in $\DD$
with $s_{-1}=\infty$ and $s_0=0$
gives rise to a continued fraction expansion
$[b_1,b_2,\cdots, b_m]$ of the terminal vertex $s_m$,
where $b_j$ is determined by the rule
explained in the above.
If we drop the condition $s_0=0$,
then $s_0\in\ZZ$ and
the edge-path determines the continued fraction expansion
of the terminal vertex $s_m$
of the form $s_0+[b_1,b_2,\cdots, b_{m}]$.

Now recall the fundamental domain
for $\RGPP{r}$
described in the introduction.
It is bounded by the four edges
$\langle \infty, 0\rangle$,
$\langle \infty, 1\rangle$,
$\langle r, r_{m-1}\rangle$
and
$\langle r, r'\rangle$,
where
\[
r_{m-1}=[a_1,a_2,\cdots,a_{m-1}]
\quad
\mbox{and}
\quad
r'=[a_1,a_2,\cdots,a_{m-1},a_m-1].
\]
Let $A_1$, $A_2$, $B_1$ and $B_2$, respectively,
be the reflections in these edges.
Then
\begin{align*}
\RGP{\infty}&=
\langle A_1\svert A_1^2=1\rangle*
\langle A_2\svert A_2^2=1\rangle,\\
\RGP{r}&=
\langle B_1\svert B_1^2=1\rangle*
\langle B_2\svert B_2^2=1\rangle.
\end{align*}
The product $A_1A_2$
is the parabolic transformation of $\DD$,
centered on the vertex $\infty$,
by $2$ units in the clockwise direction,
and it generates the
normal infinite cyclic subgroup of $\RGP{\infty}$
of index $2$.
Similarly, the product
$B_1B_2$
is the parabolic transformation of $\DD$,
centered on the vertex $\infty$,
by $2$ or $-2$ units in the clockwise
direction
according as $m$ is even or odd,
and it generates the
normal infinite cyclic subgroup of $\RGP{r}$
of index $2$.

Pick a nontrivial element, $W$, of
$\RGPP{r} =\RGP{\infty}*\RGP{r}$.
Then it
is expressed uniquely as a reduced word
$W_1W_2\cdots W_n$
or $W_0W_1\cdots W_n$
where $W_j$ is a nontrivial element of
the infinite dihedral group
$\RGP{\infty}$ or $\RGP{r}$
according as $j$ is even or odd.
When $W=W_1W_2\cdots W_n$,
we regard $W=W_0W_1\cdots W_n$ with $W_0=1$.

Set $\eta_j=+1$ or $-1$ according as
$W_j$ is orientation-preserving or reversing.
Then there is a unique integer $c_j$
such that:
\begin{enumerate}
\item
If $j$ is even, then
$W_j=(A_1A_2)^{c_j}$ or $(A_1A_2)^{c_j}A_1$
according as $\eta_j=+1$ or $-1$.
\item
If $j$ is odd, then
$W_j=(B_1B_2)^{c_j}$ or $(B_1B_2)^{c_j}B_1$
according as $\eta_j=+1$ or $-1$.
\end{enumerate}

Now let $\tilde r$ be the image of $\infty$ or $r$
by $W$.
If $n$ is odd, then $W_n\in\RGP{r}$
and hence $W(r)=W_0W_1\cdots W_{n-1}(r)$.
Similarly, if $n$ is even,
then $W(\infty)=W_0W_1\cdots W_{n-1}(\infty)$.
So, we may assume
$\tilde r =W(\infty)$ or $W(r)$ according as
$n$ is odd or even.

\begin{Lemma}
\label{Lemma:continued fraction}
Under the above setting,
$\tilde r$ has the following continued fraction expansion.
\[
\tilde r=
-2c_0+[\epsilon_1\cfr,2\epsilon_1c_1,
\epsilon_2\cfr^{-1},2\epsilon_2c_2,
\cdots,
2\epsilon_{n}c_n,
\epsilon_{n+1}\cfr^{(-1)^{n}}],
\]
where $\epsilon_j=\eta_0(-\eta_1)\cdots(-\eta_{j-1})$.
\end{Lemma}

\begin{proof}
First,
we treat the case when $W_0=1$.
Recall that $r$ is joined to $\infty$
by the edge-path
$(r_{-1},r_0,\cdots,r_{m-1},r_m)$.
Since $W_1$ fixes the point $r=r_m$,
we can join the above edge-path 
with its image by $W_1$,
and obtain the edge path
\[
(r_{-1},r_0,\cdots,r_{m-1},r_m,
W_1(r_{m-1}),\cdots, W_1(r_0), W_1(r_{-1})).
\]
This joins $\infty$ and $W_1(r_{-1})=W_1(\infty)$.
By applying the correspondence between the edge-paths
and the continued fractions,
we see that the rational number $W_1(\infty)$
has the continued fraction expansion
$[\cfr,2c_1,
-\eta_1\cfr^{-1}]$.
This can be confirmed by noticing the following facts
(see \fullref{fig.cf}).
\begin{enumerate}
\item
$W_1(r_{m-1})$ is the image of
$r_{m-1}$ by the parabolic transformation of $\DD$,
centered on the vertex $r_m=W_1(r_m)$,
by $(-1)^m2c_1$ units in the clockwise direction.
\item
$W_1(r_{j-1})$ is the image of $W_1(r_{j-1})$
by the parabolic transformation of $\DD$,
centered on the vertex $W_1(r_j)$,
by $(-1)^{j-1} a_j$ or $(-1)^{j} a_j$
units in the clockwise direction
according as $W_1$ is orientation-preserving
or reversing.
\end{enumerate}
By the temporary assumption $W_0=1$,
we have $\epsilon_1=\eta_0=+1$
and $\epsilon_2=\eta_0(-\eta_1)=-\eta_1$.
This proves the lemma when $n=1$.

Suppose $n\ge 2$.
Then,
since $W_1W_2(r_{-1})=W_1(r_{-1})$,
we can join the image of the original edge-path
by $W_1W_2$
to the above edge-path,
and obtain an edge-path
which joins $\infty$ to $W_1W_2(r)$.
More generally, by joining the images
of the original edge-path
by $1, W_1, W_1W_2,\cdots$, $W_1W_2\cdots W_n$,
we obtain an edge-path which joins $\infty$ to
$\tilde r$.
By using this edge path we obtain the lemma
for the case $W_0=1$.

Finally, we treat the case when $W_0\ne 1$.
In this case, we consider the edge-path
obtained as the image of the above edge-path by $W_0$.
Since $W_0(\infty)=\infty$,
this path joins $\infty$ to $\tilde r$
and the vertex next to $\infty$ is equal to the integer
$-2c_0$.
Hence we obtain the full assertion of the lemma.
\end{proof}

\begin{proof}
[Proof of \fullref{Prop:continued fraction}]
Immediate from
\fullref{Lemma:continued fraction}.
\end{proof}

\subsection[Presentation of K(tilde r)]{Presentation of $K(\tilde{r})$}
\label{sec.present_Ktr}

In this section,
we give a proof of \fullref{prop.hGr_Ktr}.
It is a substantially equivalent proof
to the proof of \fullref{Prop:continued fraction}
in \fullref{Sec:Continued fraction},
but written in other words
from the viewpoint of the correspondence between
$SL(2,\ZZ)$ and plat closures of $4$--braids
(see Burde--Zieschang \cite[Section 12.A]{Burde-Zieschang} for this correspondence).

In the proof of \fullref{prop.hGr_Ktr},
we use automorphisms of the modular tessellation $\DD$.
Let ${\rm Aut}(\DD)$ denote
the group of automorphisms of $\DD$,
and let ${\rm Aut}^+(\DD)$ denote
its subgroup consisting the orientation-preserving automorphisms.
Then
\begin{align*}
& {\rm Aut}^+(\DD) = PSL(2,\ZZ),
\\
& {\rm Aut}(\DD) =
\Big\{ A \in GL(2,\ZZ) \  \Big| \ {\rm det}(A) = \pm 1 \Big\} \Big/
\Big\{ \pm \mbox{\small $\begin{pmatrix} 1 & 0 \\ 0 & 1 \end{pmatrix}$} 
\Big\}.
\end{align*}

\begin{proof}[Proof of \fullref{prop.hGr_Ktr}]
We give plat presentations of $K(r)$ and $K(\tilde{r})$,
and show that they satisfy the proposition.

First, we give a plat presentation of $K(r)$, as follows.
By \fullref{Thm:Schubert},
we may assume that
$r = {\rm odd}/{\rm even}$ or ${\rm even}/{\rm odd}$.
Then
we can choose a continued fraction expansion of $r$ with even entries,
ie, of the form $[2a_1,2a_2,\cdots,2a_m]$.
Then $m$ is odd if $r = {\rm odd}/{\rm even}$,
and $m$ is even if $r={\rm even}/{\rm odd}$.
In the latter case, we replace the continued fraction expansion with
$[2a_1,\cdots,2a_{m-1}, 2a_m-1,1]$, and
set $[a'_1,a'_2,\cdots,a'_n]$ to be this continued fraction.
Namely, $[a'_1,a'_2,\cdots,a'_n]$ is
$$
\begin{cases}
[2a_1,2a_2,\cdots,2a_m]
& \mbox{ if $m$ is odd (ie, if $r={\rm odd}/{\rm even}$),} \\
[2a_1,\cdots,2a_{m-1},2a_m-1,1]
& \mbox{ if $m$ is even (ie, if $r={\rm even}/{\rm odd}$).}
\end{cases}
$$
Then we have a presentation
$$
r = B \cdot \infty ,
\mbox{ where }
B =
\mbox{\small $\begin{pmatrix} 1 & 0 \\ a'_1 & 1 \end{pmatrix}$}
\mbox{\small $\begin{pmatrix} 1 & a'_2 \\ 0 & 1 \end{pmatrix}$}
\mbox{\small $\begin{pmatrix} 1 & 0 \\ a'_3 & 1 \end{pmatrix}$} \cdots
\mbox{\small $\begin{pmatrix} 1 & 0 \\ a'_n & 1 \end{pmatrix}$} ,
$$
recalling that $B$ acts on $\QQ \cup \{ \infty \}$
by the linear fractional transformation.
Further,
the $2$--bridge link $K(r)$ is given by the plat closure
of the braid $b$ corresponding to the matrix $B$,
$$
\raisebox{12pt}[0pt][0pt]{$K(r) = \, $}
\begin{picture}(80,25)
\put(0,10){\pc{s11}{0.4}}
\put(35,12){\small$b$}
\end{picture}
\qquad
\raisebox{12pt}[0pt][0pt]{$b = \, $}
\begin{picture}(185,35)
\put(0,10){\pc{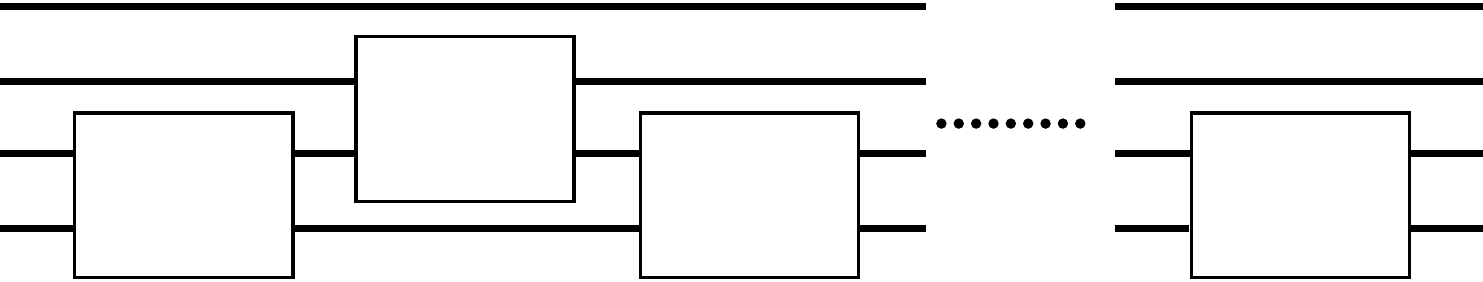}{0.4}}
\put(20,6){\small $a'_1$}
\put(48,15){\small $-a'_2$}
\put(85,6){\small $a'_3$}
\put(148,6){\small $a'_n$}
\end{picture} 
$$
where a boxed ``$a'_i$'' implies $a'_i$ half-twists.

Next, we give a plat presentation of $K(\tilde{r})$.
Since $B \in {\rm Aut}(\DD)$,
$\Gamma_r$ is presented by
$$
\Gamma_r = B \, \Gamma_\infty B^{-1},
\ \mbox{ where } \
\Gamma_\infty =
\Big\{
\mbox{\small $\begin{pmatrix} 1 & {\rm even} \\ 0 & \pm1 \end{pmatrix}$}
\Big\}
\ \subset \ {\rm Aut}(\DD) .
$$
By definition, $\tilde{r}$ belongs to the orbit of $\infty$ or
$r = B \cdot \infty$
by the action of $\hgr$,
which is generated by $\Gamma_r$ and $\Gamma_\infty$.
Hence, $\tilde{r}$ is equal to the image of $\infty$
by one of the following automorphisms of $\DD$:
\begin{align*}
&B \, \mbox{\small $\begin{pmatrix} 1 & {\rm even} \\ 0 & \pm1 
\end{pmatrix}$}
B^{-1} \mbox{\small $\begin{pmatrix} 1 & {\rm even} \\ 0 & \pm1 
\end{pmatrix}$}
\cdots
B \, \mbox{\small $\begin{pmatrix} 1 & {\rm even} \\ 0 & \pm1 
\end{pmatrix}$}
B^{-1}, \\
&
\mbox{\small $\begin{pmatrix} 1 & {\rm even} \\ 0 & \pm1 \end{pmatrix}$}
B \, \mbox{\small $\begin{pmatrix} 1 & {\rm even} \\ 0 & \pm1 
\end{pmatrix}$}
B^{-1} \mbox{\small $\begin{pmatrix} 1 & {\rm even} \\ 0 & \pm1 
\end{pmatrix}$}
\cdots
B \, \mbox{\small $\begin{pmatrix} 1 & {\rm even} \\ 0 & \pm1 
\end{pmatrix}$}
B^{-1}, \\
&
B \, \mbox{\small $\begin{pmatrix} 1 & {\rm even} \\ 0 & \pm1 
\end{pmatrix}$}
B^{-1} \mbox{\small $\begin{pmatrix} 1 & {\rm even} \\ 0 & \pm1 
\end{pmatrix}$}
\cdots
B^{-1} \mbox{\small $\begin{pmatrix} 1 & {\rm even} \\ 0 & \pm1 
\end{pmatrix}$}
B, \\
&
\mbox{\small $\begin{pmatrix} 1 & {\rm even} \\ 0 & \pm1 \end{pmatrix}$}
B \, \mbox{\small $\begin{pmatrix} 1 & {\rm even} \\ 0 & \pm1 
\end{pmatrix}$}
B^{-1} \mbox{\small $\begin{pmatrix} 1 & {\rm even} \\ 0 & \pm1 
\end{pmatrix}$}
\cdots
B^{-1} \mbox{\small $\begin{pmatrix} 1 & {\rm even} \\ 0 & \pm1 
\end{pmatrix}$}
B.
\end{align*}
By using
\begin{align*}
B_+ &:= B , \\
B_- &:=
\mbox{\small $\begin{pmatrix} 1 & 0 \\ 0 & -1 \end{pmatrix}$}
B \, \mbox{\small $\begin{pmatrix} 1 & 0 \\ 0 & -1 \end{pmatrix}$}
=
\mbox{\small $\begin{pmatrix} 1 & 0 \\ -a'_1 & 1 \end{pmatrix}$}
\mbox{\small $\begin{pmatrix} 1 & -a'_2 \\ 0 & 1 \end{pmatrix}$} \cdots
\mbox{\small $\begin{pmatrix} 1 & 0 \\ -a'_n & 1 \end{pmatrix}$},
\end{align*}
the above elements have the following unified expression:
$$
\mbox{\small $\begin{pmatrix} 1 & {\rm even} \\ 0 & 1 \end{pmatrix}$} 
B_\pm
\mbox{\small $\begin{pmatrix} 1 & 2c_1 \\ 0 & 1 \end{pmatrix}$} 
B_\pm^{-1}
\mbox{\small $\begin{pmatrix} 1 & 2c_2 \\ 0 & 1 \end{pmatrix}$} B_\pm
\mbox{\small $\begin{pmatrix} 1 & 2c_3 \\ 0 & 1 \end{pmatrix}$} \cdots
B_\pm^{\pm1}.
$$
Hence
$K(\tilde{r})$ is given by the plat closure of its corresponding braid,
$$
\raisebox{12pt}[0pt][0pt]{$K(\tilde{r}) = \, $}
\begin{picture}(334,35)
\put(0,10){\pc{s10}{0.4}}
\put(33,13){\small$b_\pm$}
\put(68,13){\small $2c_1$}
\put(103,13){\small$b_\pm^{-1}$}
\put(140,13){\small $2c_2$}
\put(178,13){\small$b_\pm$}
\put(213,13){\small $2c_3$}
\put(292,13){\small$\smash{b_\pm^{\pm1}}$}
\end{picture} 
$$
where $b_-$ is the braid corresponding to $B_-$:
$$
\raisebox{12pt}[0pt][0pt]{$b_- = \, $}
\begin{picture}(190,35)
\put(0,8){\pc{s12a}{0.4}}
\put(15,4){\small $-a'_1$}
\put(52,13){\small $a'_2$}
\put(80,4){\small $-a'_3$}
\put(144,4){\small $-a'_n$}
\end{picture}  
$$
The difference between the presentations of
the required $K(\tilde{r})$ of \fullref{prop.hGr_Ktr}
and the above $K(\tilde{r})$ is that
$b_-$ of the required $K(\tilde{r})$
is the mirror image of $b$
with respect to
(the plane intersecting this paper orthogonally along)
the central horizontal line,
while $b_-$ of the above $K(\tilde{r})$
is the mirror image of $b$ with respect to this paper.
Indeed, they are different as braids,
but their plat closures are isotopic,
because both of them are isotopic to, say,
$$
\begin{picture}(320,50)
\put(0,15){\pc{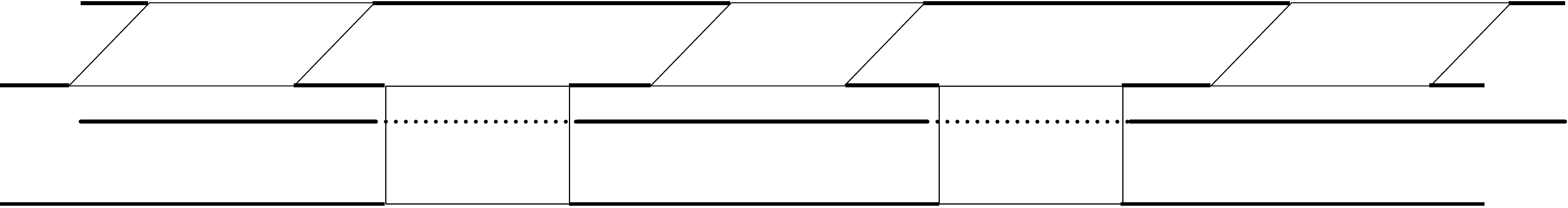}{0.4}}
\put(31,30){\fns $-2a_{m-1}$}
\put(92,8){\fns $2a_{m}$}
\put(155,30){\fns $2c_1$}
\put(199,8){\fns $\pm2a_{m}$}
\put(260,30){\fns $\mp2a_{m-1}$}
\end{picture} 
$$
and we can move any full-twists
to the opposite side of the square pillar
by an isotopy of the plat closure.
Here we draw only a part of the braid in the above figure.
(See, for example,
Siebenmann \cite{Siebenmann}, Burde--Zieschang \cite[Figure 12.9(b)]{Burde-Zieschang}
or Kauffman--Lambropoulou \cite[Section 2]{Kauffman-Lambropoulou}
for an exposition of this flype move.)
Hence, the required $K(\tilde{r})$ is isotopic to the above 
$K(\tilde{r})$,
completing the proof of \fullref{prop.hGr_Ktr}.
\end{proof}

\section[Constructing a continuous map from (S3,K(tilde r)) to (S3,K(r))]{Constructing a continuous map
$(S^3,K(\tilde r))\to(S^3,K(r))$}
\label{sec:construction1}

In this section,
we prove \fullref{thm.K_tildeK} below.
As mentioned in the introduction,
we obtain \fullref{Thm:epimorphism2}
from \fullref{prop.hGr_Ktr} and \fullref{thm.K_tildeK}.

\begin{Theorem}
\label{thm.K_tildeK}
Let $K$ be a $2$--bridge link presented by
the plat closure of a $4$--braid $b$,
and let $\tilde{K}$ be a $2$--bridge link of the form
\begin{align*}
& \raisebox{12pt}[0pt][0pt]{$K = \, $}
\begin{picture}(100,35)
\put(0,10){\pc{s11}{0.4}}
\put(37,12){\small$b$}
\end{picture}
\\
& \raisebox{12pt}[0pt][0pt]{$\tilde{K} = \, $}
\begin{picture}(330,35)
\put(0,10){\pc{s10}{0.4}}
\put(33,13){\small$b_\pm$}
\put(68,13){\small $2c_1$}
\put(103,13){\small$b_\pm^{-1}$}
\put(140,13){\small $2c_2$}
\put(178,13){\small$b_\pm$}
\put(213,13){\small $2c_3$}
\put(292,13){\small$\smash{b_\pm^{\pm1}}$}
\end{picture}
\end{align*}
for some signs of $b_\pm$ and $b_\pm^{-1}$ and for some integers $c_i$,
where
a boxed ``$2c_i$'' implies $2c_i$ half-twists,
and $\smash{b_\pm^{\pm1}}$ are the braids obtained from $b$
by mirror images in the following fashion.
$$
\begin{picture}(190,110)
\put(0,63){\pc{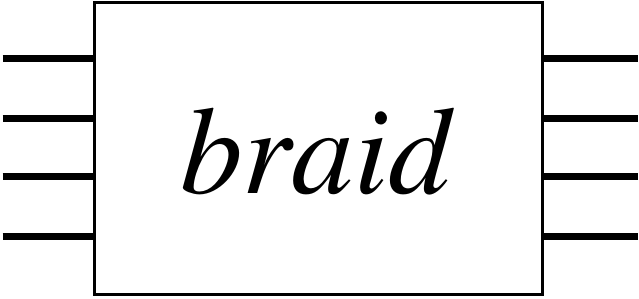}{0.35}}
\put(-50,65){\small$b = b_+ =$}
\put(95,63){\pc{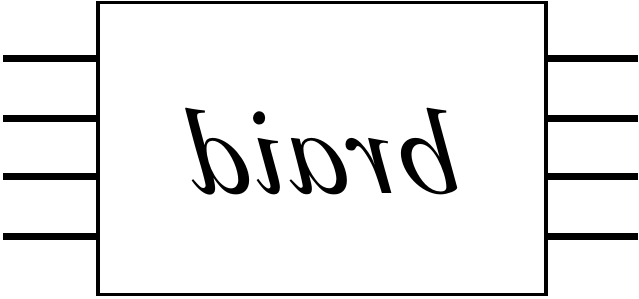}{0.35}}
\put(175,65){\small$= b_+^{-1}$}
\put(0,43){\line(1,0){190}}
\put(185,43){\vector(0,1){7}}
\put(185,43){\vector(0,-1){7}}
\put(195,41){\small \rm mirror image}
\put(85,0){\line(0,1){95}}
\put(85,90){\vector(1,0){10}}
\put(85,90){\vector(-1,0){10}}
\put(55,100){\small \rm mirror image}
\put(0,13){\pc{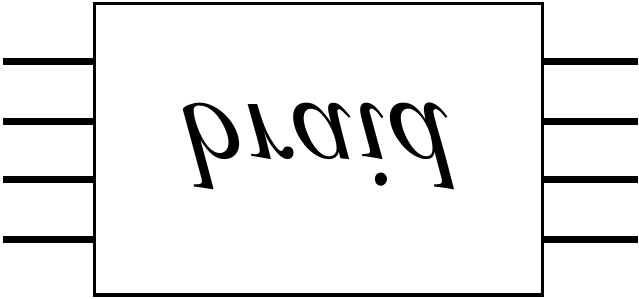}{0.35}}
\put(-30,15){\small$b_- =$}
\put(95,13){\pc{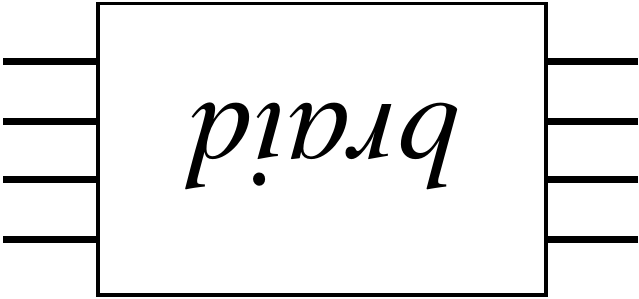}{0.35}}
\put(175,15){\small$= b_-^{-1}$}
\end{picture}
$$
Then
there is a proper branched fold map $f\co  (S^3,\tilde{K}) \to (S^3,K)$
which respects the bridge structures and
induces an epimorphism $G(\tilde{K}) \to G(K)$
\end{Theorem}

\begin{proof}
To construct the map $f$,
we partition $(S^3,K)$ and $(S^3,\tilde{K})$
into $B^3$'s and $(S^2 \times I)$'s as below,
where $I$ denotes an interval, and
we call a piece of the partition of $(S^3,\tilde{K})$
including $\smash{b_\pm^{\pm1}}$ (resp.\ $2c_i$ half-twisted strings)
a $b$--domain (resp.\ $c$--domain).
\begin{align*}
& \begin{picture}(100,40)
\put(0,10){\pc{s11}{0.4}}
\put(14,6){\pc{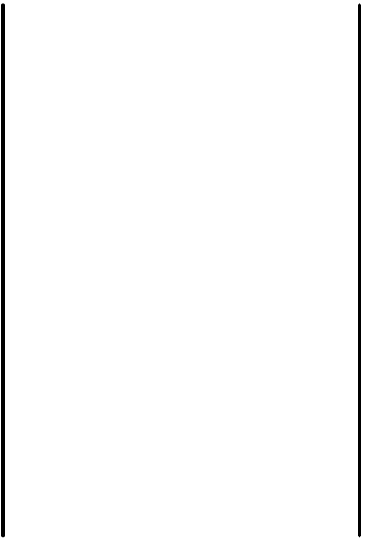}{0.4}}
\put(0,-18){\fns $B^3$}
\put(37,12){\small$b$}
\put(27,-18){\fns $S^2 \!\times\! I$}
\put(67,-18){\fns $B^3$}
\end{picture}
\\
&
\begin{picture}(330,70)
\put(0,20){\pc{s10}{0.4}}
\put(14,16){\pc{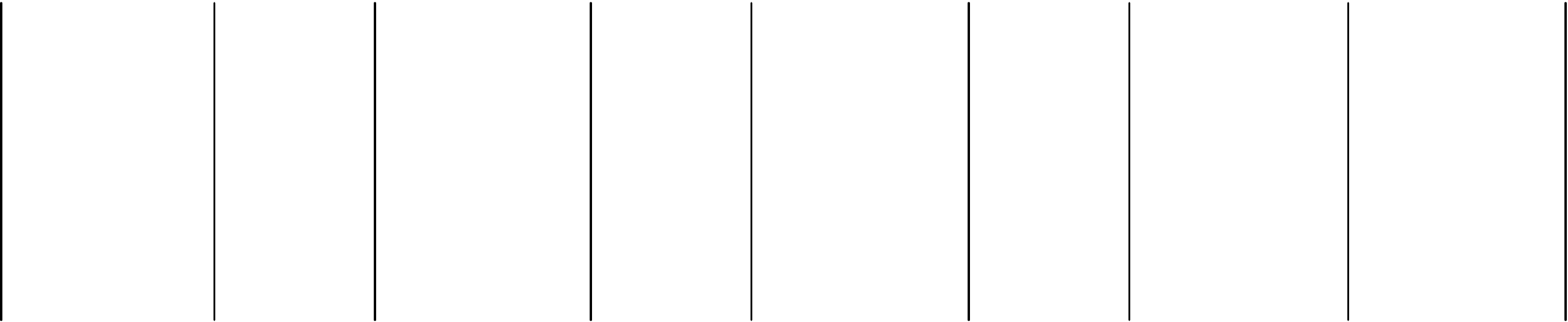}{0.4}}
\put(0,-8){\fns $B^3$}
\put(35,23){\small$b_\pm$}
\put(27,-8){\fns $S^2 \!\times\! I$}
\put(70,23){\small $2c_1$}
\put(63,-8){\fns $S^2 \!\times\! I$}
\put(105,23){\small$b_\pm^{-1}$}
\put(100,-8){\fns $S^2 \!\times\! I$}
\put(142,23){\small $2c_2$}
\put(136,-8){\fns $S^2 \!\times\! I$}
\put(180,23){\small$b_\pm$}
\put(173,-8){\fns $S^2 \!\times\! I$}
\put(215,23){\small $2c_3$}
\put(209,-8){\fns $S^2 \!\times\! I$}
\put(294,23){\small$\smash{b_\pm^{\pm1}}$}
\put(288,-8){\fns $S^2 \!\times\! I$}
\put(327,-8){\fns $B^3$}
\end{picture}
\end{align*}

We successively construct the map $f$, first on a $b$--domain,
secondly on a $c$--domain, and thirdly on $B^3$'s,
so that the required map is obtained
by gluing them together.

First, we construct $f$ on each $b$--domain
by mapping $(S^2 \!\times\! I,\, \smash{b^{\pm1}_\pm})$
to $(S^2 \!\times\! I,\, b)$
according to the definition of $\smash{b^{\pm1}_\pm}$.
To be precise, after the natural identification
of the $b$--domain and the middle piece of $(S^3,K)$
with $S^2\times I$,
the homeomorphism is given by the following self-homeomorphism on
$S^2\times I$.
\begin{enumerate}
\item
If the associated symbol is $b_+^{+1}$, the homeomorphism is $\id\times 
\id$.
\item
If the associated symbol is $b_-^{+1}$, the homeomorphism is $R_1\times 
\id$,
where $R_1$ is the homeomorphism of $S^2$ induced by
(the restriction to a level plane of)
the
reflection of $\RR^3$ in the vertical plane
which intersects this paper orthogonally
along the central horizontal line.
\item
If the associated symbol is $b_+^{-1}$, the homeomorphism is
$\id\times R_2$,
where
$R_2$ is the homeomorphism of $[-1,1]$ defined by $R_2(x)=-x$.
\item
If the associated symbol is $b_-^{-1}$, the homeomorphism is 
$R_1\times R_2$.
\end{enumerate}

Secondly, we construct the restriction of $f$ to each $c$--domain.
To this end, note that the two $b$--domains adjacent to
a $c$--domain are related either by a $\pi$--rotation
(about the vertical axis in the center of the $c$--domain)
or by a mirror reflection
(along the central level $2$--sphere
in the $c$--domain).
This follows from the following facts.
\begin{enumerate}
\item
The upper suffixes of the symbols
associated with the $b$--regions are $+1$ and $-1$
alternatively.
\item
$b_{\epsilon}^{+1}$ and $b_{\epsilon}^{-1}$
are related by a mirror reflection for each sign $\epsilon$.
\item
$b_{\epsilon}^{+1}$ and $b_{-\epsilon}^{-1}$
are related by a $\pi$--rotation for each sign $\epsilon$.
\end{enumerate}
The restriction of $f$ to a $c$--domain
is constructed as follows.
If the two relevant $b$--domains are related by a $\pi$--rotation,
then $f$ maps the $c$--domain to the left or right domain of $(S^3,K)$
as illustrated in \fullref{fig.1}.
If the two relevant $b$--domains are related by a mirror reflection,
then $f$ maps the $c$--domain to the left or right domain of $(S^3,K)$
as illustrated in \fullref{fig.2}.
In either case, the map can be made consistent with
the maps from the $b$--domains constructed in the first step.
Moreover, it is a branched fold map
and \lq\lq respects the bridge structures''.
In fact,
in the first case, it has a single branch line in the central
level $2$--sphere, whereas
in the latter case,
it has two branch lines lying in level $2$--spheres
and a single fold surface, which is actually the central level $2$--sphere.

Thirdly, the restriction of $f$ either to the
first left or to the first right domains of $(S^3,\tilde{K})$
is defined to be the natural homeomorphism to
the left or the right domain of $(S^3,K)$
which extends the map already defined on its boundary.

By gluing the maps defined on the pieces of $(S^3,\tilde{K})$,
we obtain the  desired branched fold map
$f\co (S^3,\tilde{K}) \to (S^3,K)$
which respect the bridge structures.
The induced homomorphism $f_\ast : G(\tilde{K}) \to G(K)$
maps the upper meridian pair of $G(\tilde{K})$
to that of $G(K)$ and hence
it is surjective.
\end{proof}

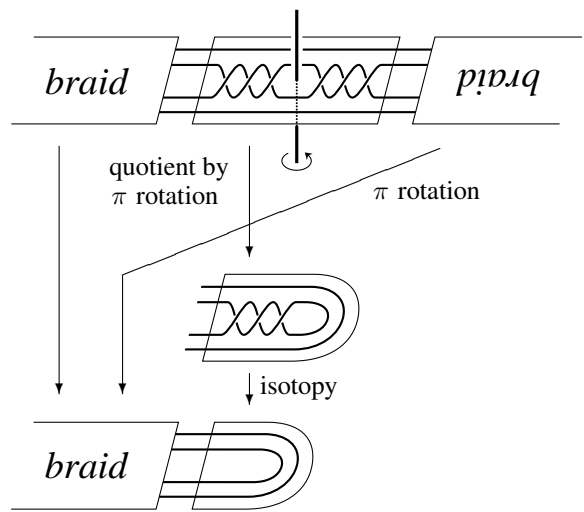
\begin{figure}[ht!]
\begin{center}
\begin{picture}(220,190)
\put(0,150){\pc{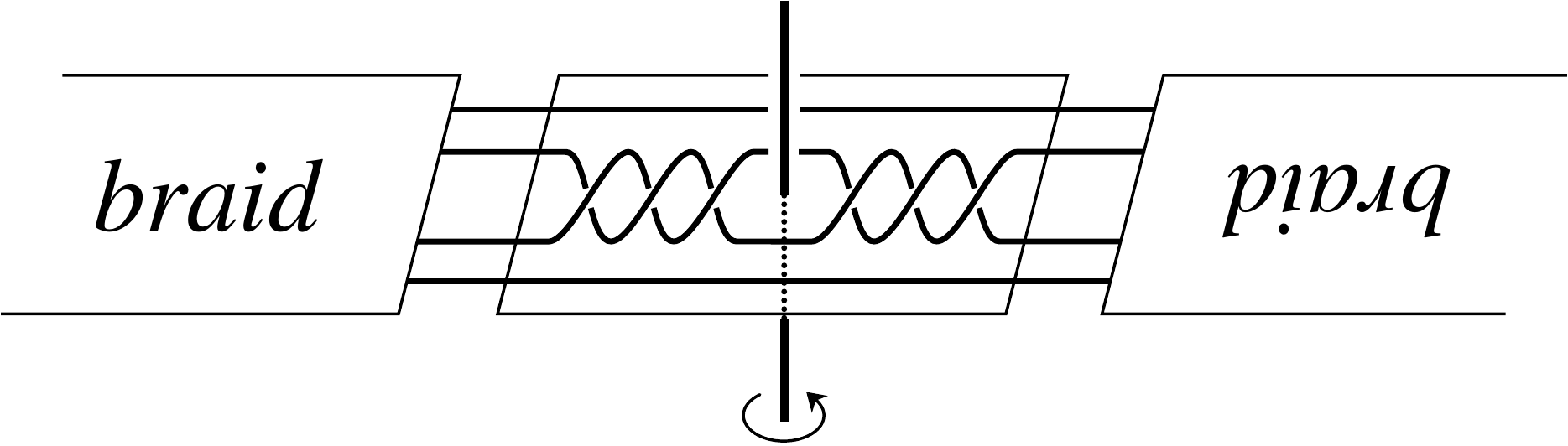}{0.4}}
\put(96,134){\vector(0,-1){42}}
\put(43,112){\shortstack[l]{\fns quotient by \\
                             \fns $\pi$ rotation}}
\put(56,64){\pc{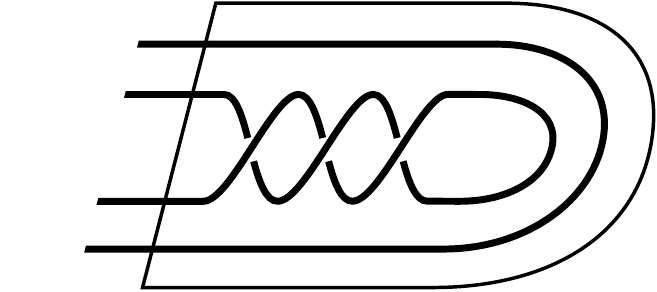}{0.4}}
\put(96,48){\vector(0,-1){12}}
\put(100,40){\fns isotopy}
\put(142,114){\fns $\pi$ rotation}
\put(48,85){\line(5,2){120}}
\put(48,85){\vector(0,-1){45}}
\put(24,134){\vector(0,-1){94}}
\put(0,8){\pc{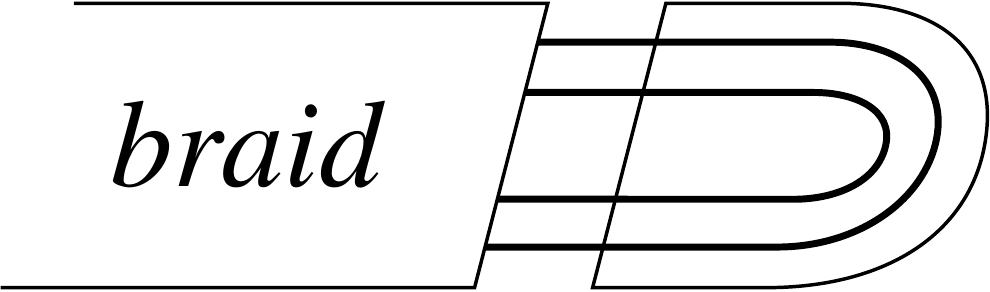}{0.4}}
\end{picture}
\end{center}
\caption{\label{fig.1}
Construction of the map $f$ on a $c$--domain,
when the two adjacent $b$--domains are related
by a $\pi$--rotation}
\end{figure}

\begin{figure}[ht!]
\begin{center}
\begin{picture}(300,216)
\put(0,172){\pc{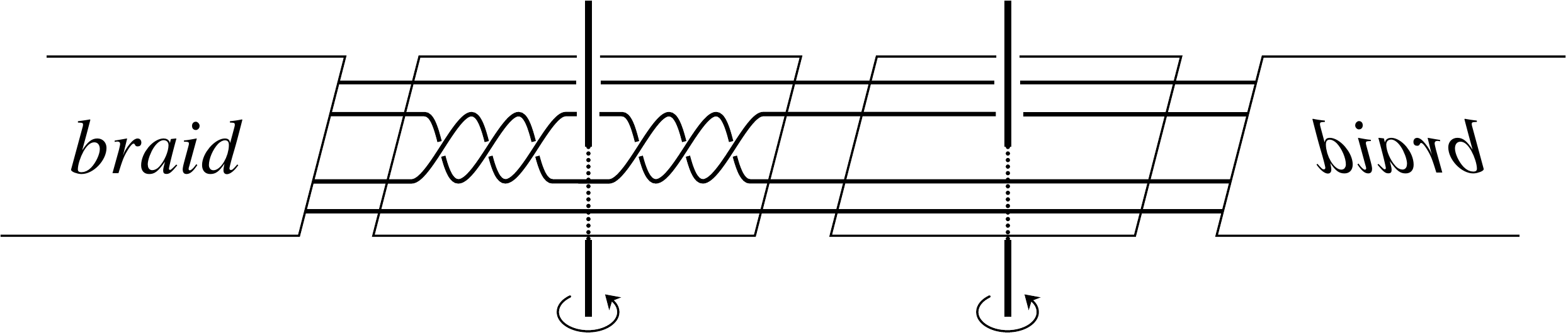}{0.4}}
\put(96,160){\vector(0,-1){70}}
\put(43,120){\shortstack[l]{\fns quotient by \\
                             \fns $\pi$ rotation}}
\put(204,160){\vector(0,-1){70}}
\put(208,104){\shortstack[l]{\fns quotient by \\
                              \fns $\pi$ rotation}}
\put(56,64){\pc{s2}{0.4}}
\put(96,48){\vector(0,-1){12}}
\put(100,40){\fns isotopy}
\put(168,64){\pc{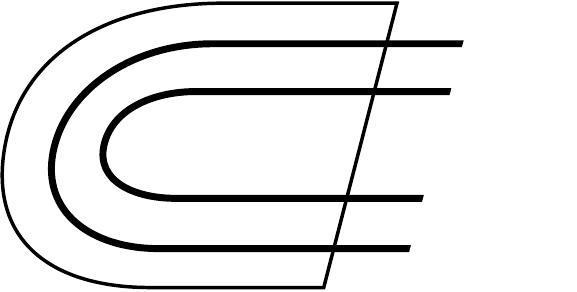}{0.4}}
\put(176,48){\vector(-2,-1){40}}
\put(156,32){\fns mirror image}
\put(117,132){\fns mirror image}
\put(48,88){\line(3,1){200}}
\put(48,88){\vector(0,-1){48}}
\put(24,160){\vector(0,-1){120}}
\put(0,8){\pc{s3}{0.4}}
\end{picture}
\end{center}
\caption{\label{fig.2}
Construction of the map $f$ on a $c$--domain,
when two adjacent $b$--domains are related by a mirror reflection}
\end{figure}
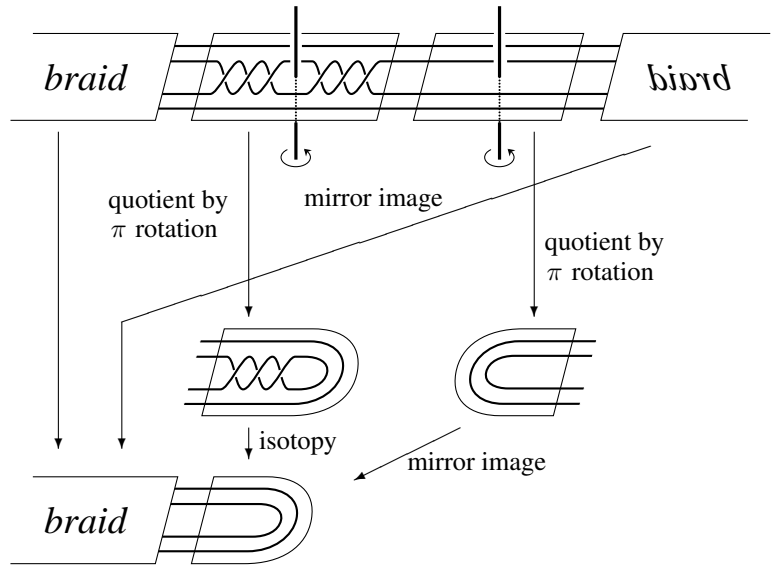

At the end of this section,
we present further properties of the map $f$
we have constructed.

\begin{Proposition}
\label{rem.f1}
The map $f\co (S^3,K(\tilde{r}))\to (S^3,K(r))$
of \fullref{thm.K_tildeK}
satisfies the following properties.
\begin{enumerate}
\item
$f$ sends the upper meridian pair of
$K(\tilde r)$ to that of $K(r)$.
\item
The degree of $f$ is equal to
$d:=\smash{\sum_j} \delta_j\epsilon_j$,
where $\epsilon_j$ and $\delta_j$ 
are the signs such that
the $j$--th $b$--domain of $\tilde K$
corresponds to $\smash{b_{\epsilon_j}^{\smash{\delta_j}}}$.

\item
The image of the longitude(s) of $K(\tilde r)$
by $f_*\co G(K(\tilde r))\to G(K(r))$ is as follows.
\begin{enumerate}
\item
If both $K(r)$ and $K(\tilde r)$ are knots,
then $f_*(\tilde\lambda)=\lambda^d$.
\item
If $K(r)$ is a knot and $K(\tilde r)$
is a $2$--component link
$\tilde K_1\cup
\tilde K_2$.
Then $f_*(\tilde\lambda_j)=
\lambda^{d/2}\mu^{\lk(\tilde K_1,\tilde K_2)}$
for each $j\in\{1,2\}$.
\item
If $K(r)$ is a $2$--component link $K_1\cup K_2$,
then $K(\tilde r)$ is also a $2$--component link
and
$f_*(\tilde\lambda_j)=
\lambda_j^{d}$
for $j\in\{1,2\}$.
\end{enumerate}
Here
$\lambda$ (resp.\ $\lambda_j$,
$\tilde \lambda$, $\tilde \lambda_j$)
denotes the longitude of the knot $K$
(resp.\ the $j$--th component of the $2$--component link $K$,
the knot $\smash{\tilde K}$,
the $j$--th component of the $2$--component link
$\tilde K_j$).
The symbol $\mu$ represents the meridian of
$K(r)$.
\item
If $\epsilon_j=+$ for every $j$,
then $f\co S^3\to S^3$ can be made to be
an $n$--fold branched covering branched over
a trivial link of $n-1$ components
which is disjoint from $K(r)$.
If $n=2$, then it is a cyclic covering.
If $n\ge 3$, then it is an irregular dihedral covering.
\end{enumerate}
\end{Proposition}

\begin{proof}
It is obvious that
the map $f$ from $(S^3,\tilde K)$
to $(S^3,K)$
constructed in the above satisfies
the conditions (1), (2) and (4).
(In order for $f$ to satisfy (4),
one may need to modify the map $f$
so that the image of the branch lines lie on
different level $2$--spheres.)
Thus we prove that $f$ satisfies the condition (3).

Suppose both $K$ and $\tilde K$ are knots.
Then the degree of the restriction of $f$ to $K(r)$
is equal to the degree $d$ of $f\co S^3\to S^3$,
and therefore
we see $f_*(\tilde \lambda)=\lambda^d\mu^c$
for some $c\in\ZZ$.
However, since $[\tilde\lambda]=0$
in $H_1(S^3-\tilde K)$,
we have  $f_*([\tilde\lambda])=0$ in $H_1(S^3-K)$,
which is the infinite cyclic group generated by $[\mu]$.
Hence $c=0$ and therefore
$f_*(\tilde \lambda)=\lambda^d$.

Suppose $K$ is a knot and $\tilde K$ is
$2$--component link $\tilde K_1\cup\tilde K_2$.
Then the degree of the restriction of $f$ to
each of the components of $\tilde K$
is equal to $d/2$,
and therefore
we see $f_*(\tilde \lambda_1)=\lambda^{d/2}\mu^c$
for some $c\in\ZZ$.
Then $\smash{[\lambda_1]=0[\mu_1]+\lk(\tilde K_1,\tilde
K_2)[\mu_2]}$ in $H_1(S^3-\tilde K)$.
Since $f_*[\mu_1]=f_*[\mu_2]=[\mu]$ in
$H_1(S^3-K)$,
we have
$f_*([\tilde \lambda_1])
=\lk(\tilde K_1,\tilde K_2)[\mu]$.
Hence we see $\smash{c=\lk(\tilde K_1,\tilde K_2)}$
and therefore
$$\smash{f_*(\tilde \lambda_1)=\lambda^{d/2}
\mu^{\smash{\lk(\tilde K_1,\tilde K_2)}}}.$$  Similarly, $\smash{f_*(\tilde \lambda_2)=\lambda^{d/2}
\mu^{\smash{\lk(\tilde K_1,\tilde K_2)}}}$.
Finally suppose both $K$ and $\tilde K$
are $2$--component links.
Then the degree of the restriction of $f$ to
each of the components of $\tilde K$
is equal to $d$,
and therefore
we see $\smash{f_*(\tilde \lambda_1)=\lambda_1^{d}\mu_1^c}$
for some $c\in\ZZ$.
Note that $f$ induces a continuous map
$\smash{S^3-\tilde K_1}\to \smash{S^3-K_1}$
and therefore
$\smash{[f_*(\tilde \lambda_1)]}=0$ in
$H_1(S^3-K_1)$.
Thus we have $c=0$ and therefore
$\smash{f_*(\tilde \lambda_1)=\lambda_1^{d}}$.
Similarly, we have
$\smash{f_*(\tilde \lambda_2)=\lambda_2^{d}}$.
\end{proof}

\begin{Remark}
\label{rem.f2}
(1)\qua Under the notation in \fullref{Prop:continued
fraction}, we can see that the degree of $f$ is equal to
$\sum_{j=1}^m \epsilon_j$.

(2)\qua Let $q'$ be the integer such that
$0<q'<p$ and $qq'\equiv 1 \pmod p$, and
set $r'=q'/p \in (0,1)$.
Then there is an orientation-preserving
self-homeomorphism of $S^3$
which sends $K(r)$ to $K(r')$
and interchanges the upper and lower bridges.
Thus \fullref{Thm:epimorphism2}
is valid even if we replace $\cfr$ with $\cfr^{-1}$.
The induced epimorphism
$f_*\co G(K(\tilde r))\to G(K(r))$ for this case
send the upper meridian pair of
$K(\tilde r)$ to a lower meridian pair.
\end{Remark}

\section{Application to character varieties}
\label{sec.icv}

In this section,
we give applications of \fullref{Thm:epimorphism1}
to character varieties of some $2$--bridge knots.

Roughly speaking,
a character variety is
(a component of a closure of) the space of conjugacy classes of
irreducible representations of the knot group $G(K)$ to
$SL(2,\CC)$. An explicit definition of the character variety is
outlined as follows; for details see Culler--Shalen \cite{Culler-Shalen} and Shalen \cite{Shalen}.
Let $R(K)$ be the space of all representations of $G(K)$ to $SL(2,\CC)$,
and let $X(K)$ be the image of the map $R(K) \to \CC^N$
taking $\rho$ to $\big( \tr(\rho(g_1)),\cdots, \tr(\rho(g_N)) \big)$
for ``sufficiently many'' $g_1,\cdots,g_N \in G(K)$.
Then $X(K)$ is shown to be an algebraic set.
We define $X^{\rm irr}(K)$
to be the Zariski closure of the image in $X(K)$ of
the space of the irreducible representations of $G(K)$ to $SL(2,\CC)$.
By a \textit{character variety} of $K$, we mean
an irreducible component of $X^{\rm irr}(K)$.
If $X^{\rm irr}(K)$ is irreducible,
$X^{\rm irr}(K)$ itself is a variety.
In fact this holds for many knots,
though in general $X^{\rm irr}(K)$ is an algebraic set
consisting of some irreducible components.

The second author \cite{Riley1,Riley_nab} concretely identified
$X^{\rm irr}(K(r))$ of any $2$--bridge knot $K(r)$
with an algebraic set in $\CC^2$
determined by a single $2$--variable polynomial,
by the map
$\rho \mapsto \big( \tr(\rho(\mu_1)), \tr(\rho(\mu_1 \mu_2^{-1})) \big)
\in \CC^2$
for the (upper or lower) meridian pair $\{ \mu_1,\mu_2 \}$
of the $2$--bridge knot group.
Further, the first author \cite{Ohtsuki_ip} classified
the ideal points of $X^{\rm irr}(K(r))$.

If $r=1/p$ for odd $p \ge 3$,
the $2$--bridge knot $K(1/p)$ is the $(2,p)$ torus knot,
and $X^{\rm irr}(K(1/p))$ consists of
$(p-1)/2$ components of affine curves \cite{Riley1},
whose generic representations are faithful
(up to the center of the torus knot group).
In particular, $X^{\rm irr}(K(1/p))$ is reducible for $p \ge 5$.
Otherwise (ie, if $K(r)$ is not a torus knot),
$K(r)$ is a hyperbolic knot,
and $X^{\rm irr}(K)$ has an irreducible component
including the faithful (discrete) representation given by
the holonomy of the complete hyperbolic structure of the knot 
complement.

We have the following application of \fullref{Thm:epimorphism1}
to the reducibility of  $X^{\rm irr}(K)$.

\begin{Corollary}
\label{Cor:character variety1}
Let $K(r)$ and $K(\tilde{r})$ be distinct nontrivial $2$--bridge knots
such that $\tilde r$ belongs to the $\RGPP{r}$--orbit of $r$ or $\infty$.
Then $X^{\rm irr}(K(\tilde r))$ is reducible.
\end{Corollary}

\begin{proof}
By \fullref{Thm:epimorphism1},
there is an epimorphism $\varphi \co G(K(\tilde r))\to G(K(r))$,
and it induces an inclusion
$\varphi^\ast \co X^{\rm irr}(K(r))\to X^{\rm irr}(K(\tilde r))$.
As mentioned above,
any $2$--bridge knot group has faithful representations
(modulo the center when it is a torus knot group),
and hence, $X^{\rm irr}(K(r))$ is nonempty.
Hence the image $\varphi^\ast \big( X^{\rm irr}(K(r)) \big)$
is a nonempty union of
the irreducible components of $X^{\rm irr}(K(\tilde{r}))$,
consisting of nonfaithful representations
$$
G(K(\tilde{r})) \xrightarrow{\rm nonfaithful}
G(K(r)) \lra SL(2,\CC) .
$$
On the other hand,
$X^{\rm irr}(K(\tilde{r}))$ has an irreducible component
including a faithful representation
$$
G(K(\tilde{r})) \xrightarrow{\rm faithful} SL(2,\CC) .
$$
(modulo the center when it is a torus knot group).
This representation is not contained in
$\varphi^\ast \big( X^{\rm irr}(K(r))\big)$,
even when $K(\tilde{r})$ is a torus knot.
Hence, $X^{\rm irr}(K(\tilde{r}))$ is reducible,
including at least 2 irreducible components.
\end{proof}

\begin{Remark}\rm
For a $2$--component $2$--bridge link $K(r)$,
the second author \cite{Riley_nab} concretely identifies
$X^{\rm irr}(K(r))$
with a $2$--dimensional algebraic set in $\CC^3$
determined by a single $3$--variable polynomial,
unless $r \in \frac12 \ZZ \cup \{ \infty \}$
(where $X^{\rm irr}(K(r))$ is empty).
Moreover, it can be shown by a similar proof that
\fullref{Cor:character variety1}
also holds for every $2$--bridge link,
unless $r \in \frac12 \ZZ \cup \{ \infty \}$.
\end{Remark}

A similar argument as the above proof is used
by Soma \cite{Soma1}
to study the epimorphisms among
the fundamental groups of hyperbolic manifolds
(see \fullref{Sec:pi_1-dominating maps}).
The proof of following corollary may be regarded as a kind of the
inverse to that of his main result in \cite{Soma1}.

\begin{Corollary}
\label{Cor:character variety2}
For any positive integer $n$,
there is a hyperbolic $2$--bridge knot $K(r)$,
such that $X^{\rm irr}(K(r))$ has at least $n$ irreducible components.
\end{Corollary}

\begin{proof}
By \fullref{Thm:epimorphism1}, we can construct an
infinite tower
\[
\cdots \to G(K(r_n)) \to G(K(r_{n-1})) \to
\cdots \to G(K(r_2)) \to G(K(r_1))
\]
of epimorphisms among $2$--bridge knot groups
such that none of the epimorphisms is an isomorphism.
Then by the argument in the proof
of \fullref{Cor:character variety1},
$X^{\rm irr}(K(r_n))$ has an irreducible component
including a representation
$$
G(K(r_n)) \lra
G(K(r_i)) \xrightarrow{\rm faithful} SL(2,\CC) ,
$$
for each $i = 1,2,\cdots,n$.
Since these components are distinct,
$X^{\rm irr}(K(r_n))$ has at least $n$ irreducible components.
\end{proof}

\section[Application to pi1-surjective maps between 3-manifolds]{Application to
$\pi_1$--surjective maps between $3$--manifolds}
\label{Sec:pi_1-dominating maps}

Let $M$ and $N$ be connected closed orientable
$3$--manifolds.
Then a continuous map $f\co M\to N$
is said to be \textit{$\pi_1$--surjective}
if $f_*\co \pi_1(M)\to\pi_1(N)$
is surjective.
If the degree $d$ of $f$ is nonzero,
then the index
$[\pi_1(N): f_*(\pi_1(M))]$
is a divisor of $d$.
In particular, if the degree of $f$ is $1$, then $f$
is $\pi_1$--surjective.
Motivated by this fact,
Reid--Wang--Zhou
\cite{Reid-Wang-Zhou}
proposed various questions,
in relation with
Simon's problem \cite[Problem 1.12]{Kirby}
and
Rong's problem \cite[Problem 3.100]{Kirby}.
In this section, we study the following questions
proposed in \cite{Reid-Wang-Zhou}.
\begin{enumerate}
\item
(Question 1.5)\qua
Let $M$ and $N$ be closed aspherical $3$--manifolds
such that the rank of $\pi_1(M)$ equals the rank of $\pi_1(N)$.
Assume $\phi\co \pi_1(M)\to\pi_1(N)$ is surjective or
its image is a subgroup of finite index.
Does $\phi$ determine a map $f\co M\to N$ of nonzero degree?
\item
(Question 3.1(D))\qua
Are there only finitely many
closed orientable $3$--manifolds $M_i$ with the same
first Betti number, or the same $\pi_1$--rank,
as that of
a closed orientable $3$--manifold $M$,
for which there is an epimorphism
$\pi_1(M)\to \pi_1(M_i)$?
\end{enumerate}
Example 1.4 in \cite{Reid-Wang-Zhou}
presents a closed hyperbolic
$3$--manifold $M$ with $\pi_1$--rank $>2$
(and $b_1(M)>2$),
for which there are infinitely many mutually nonhomeomorphic
hyperbolic $3$--manifolds $M_i$
with $\pi_1$--rank $2$ (and hence $b_1(M)\le 2$),
such that there is a
$\pi_1$--surjective degree $0$ map $M\to M_i$.
This shows that the conditions on
the $\pi_1$--ranks (and Betti numbers) in the above questions are
indispensable.
Moreover, they give the following partial positive answers
to the questions for Seifert fibered spaces and non-Haken manifolds.
\begin{enumerate}
\item
Any $\pi_1$--surjective map between
closed orientable Seifert fibered spaces
with the same $\pi_1$--rank and with orientable base orbifolds
is of nonzero degree \cite[Theorem 2.1]{Reid-Wang-Zhou}.
\item
For any non-Haken closed orientable hyperbolic manifold $M$,
there are only finitely many closed orientable
hyperbolic $3$--manifolds $M_i$ for which there
is an epimorphism $\pi_1(M)\to\pi_1(M_i)$
\cite[Theorem 3.6]{Reid-Wang-Zhou}.
\end{enumerate}
Gonzal\'ez-Ac\~una and Ram\'inez
have constructed
a counter example to the questions
where
the source manifold is hyperbolic and
the target manifolds are Seifert fibered spaces
\cite[Example 26]{Gonzalez-Raminez2}.
They asked if there is a counter example
where
the source and target manifolds
are hyperbolic manifolds.
The following corollary to \fullref{Thm:epimorphism2}
gives such an example.

\begin{Corollary}
\label{Cor:Dominante infinite mfds}
There is a closed orientable hyperbolic Haken
$3$--manifold
$M$  and infinitely many mutually nonhomeomorphic,
closed, orientable, hyperbolic $3$--manifolds $M_i$
which satisfy the following conditions.
\begin{enumerate}
\item
There is a $\pi_1$--surjective degree $0$ map
$f_i\co M\to M_i$.
\item
The ranks of the fundamental groups of $M$ and $M_i$
are all equal to $2$.
\end{enumerate}
\end{Corollary}

\begin{proof}
Pick a proper map
$f\co (S^3,K(\tilde r))\to
(S^3,K(r))$ between $2$--bridge links
satisfying the conditions of
\fullref{Thm:epimorphism2},
such that
the degree of $f$ is $0$ and
$K(\tilde r)$ is
a $2$--component link
$\tilde K_1\cup \tilde K_2$
and $K(r)$ is a knot.
Set $q=\lk(\tilde K_1,\tilde K_2)$.
Then, by \fullref{rem.f1}, $f$ maps the essential simple loop
$\tilde \lambda_j-q\tilde\mu_j$ on
$\partial N(\tilde K_j)$
to a nullhomotopic loop on $\partial N(K)$.
Let $M_0$ be the manifold obtained
by surgery along the link
$\tilde K_1\cup \tilde K_2$,
where $2$--handle is attached along the curve
$\tilde \lambda_j-q\tilde\mu_j$
on $\partial N(\tilde K_j)$ for each $j=1,2$.
Then for every manifold $M(s)$ ($s\in\smash{\QQQ}$),
obtained by $s$--surgery on $K(r)$,
the map $f\co S^3-K(\tilde r)\to S^3-K(r)$
extends to $\pi_1$--surjective map
$M_0\to M(s)$ of degree $0$.
On the other hand, we may choose
$r$ and $\tilde r$ so that
$M_0$ is hyperbolic
and that $M(s)$ are hyperbolic with finite exceptions.
For example, if $r=[2,2]$ and
$\tilde r=[2,2,2,-2,-2]$,
then $K(r)$ is the (hyperbolic) figure-eight knot
and therefore $M(s)$ is hyperbolic with finite exceptions.
Moreover we can check by SnapPea \cite{Weeks}
that $M_0$ is hyperbolic.
Since $M_0$ and $M(s)$ are closed hyperbolic manifolds
whose fundamental groups are generated by two elements,
their $\pi_1$--ranks must be equal to $2$.
Moreover, $M$ is Haken, because
the first Betti number of $M$ is $2$ or $1$
according as
$\lk(\tilde K_1,\tilde K_2)=0$ or not.
This completes the proof of
\fullref{Cor:Dominante infinite mfds}.
\end{proof}

\begin{Remark}\rm
In the above corollary,
the first Betti number of $M$ is $\ge 1$,
whereas
the first Betti numbers of $M_i$
are all equal to $0$.
This is the same for \cite[Example
26]{Gonzalez-Raminez2}.
We do not know if there is a counter
example to the second question such that
the first Betti numbers of $M$ and $M_i$
are all equal.
\end{Remark}

\fullref{Cor:Dominante infinite mfds}
does not have a counterpart
where the condition that
the maps are of degree $0$
are replaced with the condition that
the maps are nonzero degree.
In fact, Soma \cite{Soma1} proves that
for every closed, connected, orientable $3$--manifold $M$,
the number of mutually nonhomeomorphic,
orientable, hyperbolic $3$--manifolds dominated by $M$
is finite.
(Here a $3$--manifold $N$ is said to be \textit{dominated} by $M$
if there is exists a nonzero degree map
$f\co M\to N$.)
In Soma's theorem, the condition that the manifolds are orientable
is essential.
In fact, as is noted in \cite[Introduction]{Soma1},
some arguments in
Boileau--Wang \cite[Section 3]{Boileau-Wang}
implies that there is a nonorientable
manifold $M$ which dominates
infinitely many mutually nonhomeomorphic $3$--manifolds.

We present yet another application
of \fullref{Thm:epimorphism2}
to $\pi_1$--surjective maps.
By studying the character varieties,
Soma \cite{Soma2} observed that
there is no infinite descending tower
of $\pi_1$--surjective maps between
orientable hyperbolic $3$--manifolds,
namely,
any infinite sequence of
$\pi_1$--surjective maps
\[
M_0\to M_1\to \cdots \to
M_i \to M_{i+1}\to \cdots
\]
between orientable
hyperbolic $3$--manifolds $M_i$ (possibly of infinite
volume) contains an isomorphism.
On the other hand,
Reid-Wang-Zhou constructed an infinite ascending tower
of $\pi_1$--surjective maps of degree $>1$
between closed orientable hyperbolic $3$--manifolds
with the same $\pi_1$--rank
\cite[Example 3.2]{Reid-Wang-Zhou}.
The following corollary to \fullref{Thm:epimorphism2}
refines their example,
by constructing such a tower for degree $1$ maps.

\begin{Corollary}
\label{Thm:infinite tower}
There is an infinite ascending tower of
$\pi_1$--surjective maps of degree $1$
\[ \cdots \to M_i \to M_{i-1}\to \cdots \to M_1\to
M_0\]
between
closed (resp.\ cusped) orientable hyperbolic
$3$--manifolds
which satisfies the following conditions.
\begin{enumerate}
\item
The ranks of $\pi_1(M_i)$ are all equal to $2$.
\item
$H_1(M_i)\cong \ZZ$ for every $i$, and
each map induces an isomorphism between the homology groups.
\end{enumerate}
\end{Corollary}

\begin{proof}
By \fullref{Thm:epimorphism2},
we can construct an infinite ascending tower
\[
\cdots \to (S^3,K_i) \to (S^3,K_{i-1}) \to
\cdots \to (S^3, K_1) \to (S^3,K_0)
\]
of degree $1$ proper maps among hyperbolic $2$--bridge knots,
such that each map induces an epimorphism
among the knot groups which is not an isomorphism.
By taking the knot complements and induced maps,
we obtain a desired tower of cusped hyperbolic manifolds.
Now, let $M_i$ be the result of $0$--surgery on $K_i$.
Since each map sends the meridian-longitude pair
of $K_i$ to that of $K_{i-1}$,
the above tower gives rise to a tower
of $\pi_1$--surjective maps of degree $1$
\[
\cdots \to M_i \to M_{i-1} \to
\cdots \to M_1 \to M_0.
\]
By the classification
of exceptional surgeries on hyperbolic
$2$--bridge knots
due to
Brittenham--Wu
\cite{Brittenham-Wu}
and by the orbifold theorem
\cite{Boileau-Porti,Cooper-Hodgson-Kerckhoff}, we can choose
$K_i$ so that every $M_i$ is hyperbolic.
Thus the above tower satisfies the conditions
on the $\pi_1$--rank and the homology.
Thus our remaining task is to show that
none of the maps is a homotopy equivalence.
To this end, we choose $K_i$ so that
the genus of $K_i$ is monotone increasing.
This can be achieved by starting from
the continued fraction consisting of only nonzero even integers,
ie, the components of the sequence $\cfr$ in
\fullref{Prop:continued fraction} are all
nonzero even integers.
(Though it seems that any tower satisfies this condition,
it is not totally obvious that this is actually the case.)
Then the degree of the Alexander polynomial
of $K_i$ is monotone increasing.
Since the Alexander polynomial is an invariant
of (the homotopy type of)
the manifold obtained by $0$--surgery,
this implies that $M_i$ are mutually non--homotopy equivalent.
This completes the proof of
\fullref{Thm:infinite tower}.
\end{proof}

We note that
if we drop the condition on $\pi_1$--rank,
then the existence of such
an infinite ascending tower is obvious
from Kawauchi's imitation theory \cite{Kawauchi}.

\section{Some questions}
\label{sec:questions}

In this final section,
we discuss two questions
related to \fullref{Thm:epimorphism1} and \fullref{Thm:epimorphism2}.

\begin{Question}\label{Question:nullhomotopic}
(1)\qua Does the converse to \fullref{Thm:epimorphism1}
holds?
Or more generally,
given a $2$--bridge link $K(r)$,
which $2$--bridge link group has
an epimorphism
onto the link group of $K(r)$?

(2)\qua Does the converse to
\fullref{Cor:nullhomotopic} hold?
Namely, is it true that
$\alpha_s$ is nullhomotopic in $S^3-K(r)$
if and only if $s$ belongs to the $\RGPP{r}$--orbit of $\infty$ or $r$?
\end{Question}

F\,Gonzal\'ez-Ac\~una and A\,Ram\'inez
gave a nice partial answer to
the first question.
They proved that
if $r=1/p$ for some odd integer $p$,
namely $K(r)$ is a $2$--bridge torus knot,
then the knot group of a $2$--bridge knot $K(\tilde r)$
($\tilde r=\tilde q/\tilde p$ for some
odd integer $\tilde p$)
has an epimorphism onto the knot group
of the $2$--bridge torus knot $K(1/p)$
if and only if $\tilde r$ has a continued fraction
expansion of the form in
\fullref{Prop:continued fraction}, namely
$\tilde r$ is contained in the $\RGPP{r}$--orbit of
$r$ or $\infty$.
(See Gonz{\'a}lez-Acu{\~n}a--Ram{\'{\i}}rez \cite[Theorem 1.2]{Gonzalez-Raminez1}
and \cite[Theorem 16]{Gonzalez-Raminez2}).
By the proof of \fullref{Thm:epimorphism1},
this also implies a partial positive answer to
the second question when
$r=1/p$ for some odd integer $p$.

In \cite{Sakuma},
the last author studied the second question,
in relation with a possible variation
of McShane's identity \cite{McShane}
for $2$--bridge links,
by using Markoff maps,
or equivalently, trace functions
for \lq\lq type-preserving'' $SL(2,\CC)$--representations
of the fundamental group
of the once-punctured torus.
See Bowditch \cite{Bowditch2} for the precise definition
and detailed study of Markoff maps,
and our joint paper \cite[Section 5.3]{ASWY2} for
the relation of Markoff maps and the $2$--bridge links.
He announced
an affirmative answer to the second question
for the $2$--bridge torus link
$K(1/p)$ for every integer $p$,
the figure-eight knot
$K(2/5)$ and the $5_1$--knot $K(3/7)$.
In his master thesis \cite{Eguchi}
supervised by the third author,
Tomokazu Eguchi obtained,
by numerical calculation of Markoff maps,
an affirmative answer to the question
for the twist knots $K(n/(2n+1))$
for $2\le n\le 10$.

At the beginning of the introduction,
we mentioned the problem:
for a given knot $K$,
characterize a knot $\tilde{K}$ which admits
an epimorphism $G(\tilde{K}) \to G(K)$.
Motivated by \fullref{Thm:epimorphism2},
we consider the following procedure
to construct knots $\tilde{K}$ from a given knot $K$.
\begin{itemize}
\item[(a)]
Choose a branched fold map $f\co  M \to S^3$
for a closed $3$--manifold $M$ such that
the image of each component of the fold surface is transverse to $K$
and the image of each component of branch curve is a knot disjoint from 
$K$.
Then we obtain $\tilde{K}$ as the preimage $f^{-1}(K)$.
\end{itemize}
Such a $\tilde{K}$ often admits an epimorphism $G(\tilde{K}) \to G(K)$.
Further, when we have an epimorphism $\phi\co G(\tilde{K}) \to G(K)$,
we consider the following procedures to modify $(M,\tilde{K})$.
\begin{itemize}
\item[(b)]
Replace $(M,\tilde{K})$ with the pair obtained from $(M,\tilde{K})$
by surgery along a simple closed curve in the kernel of $\phi$.
\item[(c)]
Replace $\tilde{K}$ by the following move.
$$
\begin{picture}(180,50)
\put(0,15){\pc{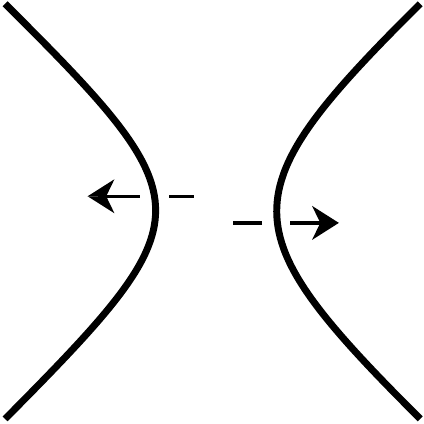}{0.4}}
\put(4,18){\fns $g_1$}
\put(47,15){\fns $g_2$}
\put(70,20){\vector(1,0){80}}
\put(75,5){\fns if $\phi(g_1)=\phi(g_2)$}
\put(155,15){\pc{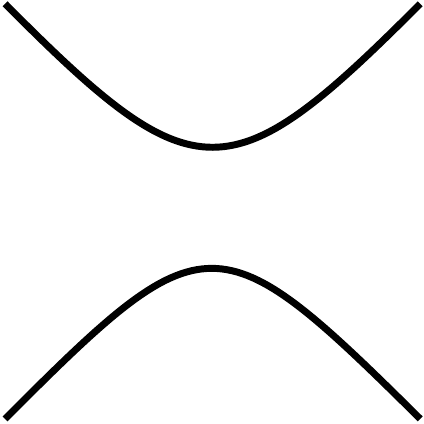}{0.4}}
\end{picture}
$$
\end{itemize}
We can construct many examples of $\tilde{K}$ from $K$
by the construction (a),
further modifying $\tilde{K}$
by applying (b) and (c) repeatedly.
(Even if an intermediate ambient $3$--manifold is not $S^3$,
we may obtain a knot in $S^3$ by modifying it
into $S^3$ by using  (b).)
The following question asks whether
the constructions (a), (b) and (c)
give a topological characterization of a knot $\tilde{K}$
having an epimorphism $G(\tilde{K}) \to G(K)$
for a given knot~$K$.

\begin{Question}
\label{q.const_tK}
If there is an epimorphism $G(\tilde{K}) \to G(K)$
between knot groups preserving the peripheral structure,
can we obtain $\tilde{K}$ from $K$
by repeatedly applying the above constructions
{\rm (a), (b)} and {(c)}?
\end{Question}

The first author has given a positive answer to
this question for all such pairs of prime knots $(\tilde{K},K)$
with up to 10 crossings,
by checking the list in Kitano--Suzuki \cite{Kitano-Suzuki}
(see \fullref{tbl.const_tK}).
The answer to the question is also positive,
if either (i) $\tilde{K}$ is a satellite knot with pattern knot $K$
(cf \cite[Proposition 3.4]{Silver-Whitten}),
or  (ii) $\tilde{K}$ is a satellite knot of $K$
of degree $1$
(ie, $\tilde{K}$ is homologous to $K$
in the tubular neighborhood of $K$.)
In particular,
the answer to \fullref{q.const_tK} is positive,
when $\tilde{K}$ is a connected sum of $K$ and some knot.
We can also obtain a positive answer for the case 
when there are a ribbon concordance $C$ from $\tilde{K}$ to $K$
and an epimorphism $G(\tilde{K}) \to G(K)$
which is compatible with
$G(\tilde{K}) \to \pi_1(S^3 \!\times\! I - C) \leftarrow G(K)$
(cf \cite[Lemma 3.1]{Gordon}).
(In general, a ribbon concordance between knots does not necessarily
induce an epimorphism between their knot groups; see \cite{Miyazaki}.)

Finally, we note that
certain topological interpretations of
some of the epimorphisms in the table of \cite{Kitano-Suzuki}
have been obtained by
Kitano--Suzuki \cite{Kitano-Suzuki2},
from a different view point.

\begin{table}[p!]
    \renewcommand{\arraystretch}{1.125}
\begin{center}
\begin{tabular}{l}
\small $8_5 \ \approx \ 3_1 \# 3_1 \lra 3_1$ \\
\small $8_{10} \ \approx \ 3_1 \# \ol{3_1} \lra 3_1$ \\
\small $8_{15} \ \approx \ 3_1 \# 3_1 \lra 3_1$ \\
\small $8_{18} \ \approx \ 3_1$ \\
\small $8_{19} \ \approx \ \ol{3_1} \# \ol{3_1} \lra \ol{3_1}$ \\
\small $8_{20} \ \approx \ 3_1 \# \ol{3_1} \lra 3_1$ \\
\small $8_{21} \ \approx \ 3_1 \# 3_1 \lra 3_1$ \\
\small $9_1 \, = \, K(\mbox{\fns $[9]$}) \lra 3_1$ \\
\small $9_6 \, = \, K(\mbox{\fns $[6,-2,3]$}) \lra 3_1$ \\
\small $9_{16} \ \approx \ \ol{3_1} \# \ol{3_1} \lra \ol{3_1}$ \\
\small $9_{23} \, = \, K(\mbox{\fns $[-3,2,-3,2,-3]$}) \lra 3_1$ \\
\small $9_{24} \ \approx \ 3_1 \# \ol{3_1} \lra 3_1$ \\
\small $9_{28} \ \approx \ 3_1 \# 3_1 \lra 3_1$ \\
\small $9_{40} \lra 3_1$ \\
\small $10_{5} \, = \, K(\mbox{\fns $[3,-2,6]$}) \lra 3_1$ \\
\small $10_{9} \, = \, K(\mbox{\fns $[3,2,-6]$}) \lra 3_1$ \\
\small $10_{32} \, = \, K(\mbox{\fns $[3,-2,3,-2,-3]$}) \lra 3_1$ \\
\small $10_{40} \, = \, K(\mbox{\fns $[3,-2,-3,2,-3]$}) \lra 3_1$ \\
\small $10_{61} \ \approx\approx \ \ol{3_1} \# \ol{3_1} \lra \ol{3_1}$ 
\\
\small $10_{62}\ \approx\approx \ 3_1 \# \ol{3_1} \lra 3_1$ \\
\small $10_{63} \ \approx\approx \ 3_1 \# 3_1 \lra 3_1$ \\
\small $10_{64} \ \approx\approx \ \ol{3_1} \# \ol{3_1} \lra \ol{3_1}$ 
\\
\small $10_{65}\ \approx\approx \ 3_1 \# \ol{3_1} \lra 3_1$ \\
\small $10_{66} \ \approx\approx \ 3_1 \# 3_1 \lra 3_1$ \\
\small $10_{76} \ \approx \ \ol{3_1} \# \ol{3_1} \lra \ol{3_1}$ \\
\small $10_{77}\ \approx \ 3_1 \# \ol{3_1} \lra 3_1$ \\
\small $10_{78} \ \approx \ 3_1 \# 3_1 \lra 3_1$ \\
\small $10_{82}\ \approx \ 3_1 \# \ol{3_1} \lra 3_1$ \\
\small $10_{84}\ \approx \ 3_1 \# \ol{3_1} \lra 3_1$
\end{tabular}
\qquad\quad
\begin{tabular}{l}
\small $10_{85} \ \approx \ 3_1 \# 3_1 \lra 3_1$ \\
\small $10_{87}\ \approx \ 3_1 \# \ol{3_1} \lra 3_1$ \\
\small $10_{98} \ \approx\approx \ 3_1 \# 3_1 \lra 3_1$ \\
\small $10_{99}\ \approx\approx \ 3_1 \# \ol{3_1} \lra 3_1$ \\
\small $10_{103}\ \sim\sim \ 3_1$ \\
\small $10_{106}\ \approx \ \ol{3_1}$ \\
\small $10_{112}\ \approx\approx\approx \ \ol{3_1}$ \\
\small $10_{114}\ \approx \ \ol{3_1}$ \\
\small $10_{139}\ \approx\approx \ \ol{3_1}$ \\
\small $10_{140}\ \approx\approx\ 3_1 \# \ol{3_1} \lra 3_1$ \\
\small $10_{141} \ \approx\approx \ 3_1 \# 3_1 \lra 3_1$ \\
\small $10_{142}\ \approx\approx\ 3_1 \# \ol{3_1} \lra 3_1$ \\
\small $10_{143}\ \approx\approx\ 3_1 \# \ol{3_1} \lra 3_1$ \\
\small $10_{144} \ \approx\approx \ 3_1 \# 3_1 \lra 3_1$ \\
\small $10_{159}\ \approx\approx \ 3_1$ \\
\small $10_{164}\ \approx \ \ol{3_1}$ \\
\small $8_{18} \lra 4_1$ \\
\small $9_{37} \ \sim \ 4_1 \# 4_1 \# 4_1 \lra 4_1$ \\
\small $9_{40}\ \approx \ 4_1$ \\
\small $10_{58} \ \approx \ 4_1 \# 4_1 \lra 4_1$ \\
\small $10_{59} \ \approx \ 4_1 \# 4_1 \lra 4_1$ \\
\small $10_{60} \ \approx \ 4_1 \# 4_1 \lra 4_1$ \\
\small $10_{122} \lra 4_1$ \\
\small $10_{136} \ \sim \ 8_6^2 \lra 4_1$ \\
\small $10_{137} \ \approx \ 4_1 \# 4_1 \lra 4_1$ \\
\small $10_{138} \ \approx \ 4_1 \# 4_1 \lra 4_1$ \\
\small $10_{74} \ \sim\sim \ \ol{5_2}$ \\
\small $10_{120} \lra 5_2$ \\
\small $10_{122}\ \approx \ \ol{5_2}$
\end{tabular}
\end{center}
\vskip8pt
\caption{\label{tbl.const_tK}
Sketch answer to \fullref{q.const_tK}
for the pairs of prime knots
with up to 10 crossings,
listed by Kitano-Suzuki \protect\cite{Kitano-Suzuki}.
Here, we denote the procedures (a), (b), (c) by
arrow, ``$\approx$'', ``$\sim$'' respectively,
and, say ``$\approx \approx$''
means to apply (b) twice.
The numerical notation for knots and links is the one in Rolfsen \protect\cite{Rolfsen},
and $\ol{K}$ denotes the mirror image of $K$.}
\end{table}

\bibliographystyle{gtart}
\bibliography{link}

\end{document}